\documentclass[11pt,reqno,letterpaper]{amsart}
\usepackage{latexsym}
\usepackage{amsgen,amsmath,amsxtra,amsfonts,amsthm}
\usepackage{mathabx} 
\usepackage[latin1]{inputenc}
\usepackage{color}
\usepackage[active]{srcltx}
\usepackage{hyperref}
\usepackage{fullpage} 

\usepackage{graphicx}

\newcommand{\Om}{\Omega}
\newcommand{\R}{\mathbb{R}}

\theoremstyle{plain}
\newtheorem{theorem}{Theorem}
\newtheorem{corollary}[theorem]{Corollary}
\newtheorem{lemma}[theorem]{Lemma}
\newtheorem{prop}[theorem]{Proposition}

\theoremstyle{definition}

\theoremstyle{remark}
\newtheorem{remark}[theorem]{Remark}

\numberwithin{equation}{section}
\numberwithin{theorem}{section}

\begin{document}
\title{The Gelfand problem for the Infinity Laplacian}

\author[F. Charro]{Fernando Charro}
\address{Department of Mathematics, Wayne State University, 656 W. Kirby, Detroit, MI 48202, USA}
\email{fcharro@wayne.edu}
\thanks{}

\author[B. Son]{Byungjae Son}
\address{Department of Mathematics and Statistics, University of Maine, Orono, ME 04469, USA}
\email{byungjae.son@maine.edu}

\author[P. Wang]{Peiyong Wang}
\address{Department of Mathematics, Wayne State University, 656 W. Kirby, Detroit, MI 48202, USA}
\email{av2308@wayne.edu}

\thanks{F.C. partially supported by grant MTM2017-84214-C2-1-P and 
PID2019-110712GB-I100 funded by MCIN/AEI/10.13039/501100011033 and by ``ERDF A way of making Europe''.
}




\begin{abstract}
We study the asymptotic behavior as $p\to\infty$ of the Gelfand problem
\begin{equation*}
\left\{
\begin{aligned}
-&\Delta_{p} u=\lambda\,e^{u}&&\textrm{in}\ \Omega\subset\R^n\\
&u=0 &&\textnormal{on}\ \partial\Omega.
\end{aligned}
\right.
\end{equation*}
Under an appropriate rescaling on $u$ and $\lambda$, we  prove uniform convergence of solutions of the Gelfand problem to solutions of 
\[
\left\{
\begin{aligned}
&\min\left\{|\nabla{}u|-\Lambda\,e^{u},
-\Delta_{\infty}u\right\}=0&&\textrm{in}\ \Om,\\
&u=0\ &&\text{on}\ \partial\Omega.
\end{aligned}
\right.
\]
We discuss  existence, non-existence, and  multiplicity of  solutions of the limit problem in terms of  ~$\Lambda$.
\end{abstract}

\maketitle



\begin{center}
\it\large
Dedicated to the memory of Ireneo Peral,  with love and admiration
\end{center}

\vspace{10pt}

\section{Introduction}

We are interested in the asymptotic behavior as $p\to\infty$ of sequences of solutions of the problem
\begin{equation}\label{intro.original}
\left\{
\begin{aligned}
-&\Delta_{p} u=\lambda\,e^{u}&&\textrm{in}\ \Omega\subset\R^n\\
&u=0 &&\textnormal{on}\ \partial\Omega.
\end{aligned}
\right.
\end{equation}
In the case $p=2$, problem \eqref{intro.original} is  known as the
 Liouville-Bratu-Gelfand problem \cite{Bratu, Gelfand,Liouville}; see also \cite{Davila,Joseph.Lundgren.1973}. It appears in connection with 
prescribed Gaussian curvature problems \cite{Chanillo-Kiessling,Liouville}, emission of electricity from hot bodies \cite{Richardson.1921}, and the
equilibrium of gas spheres and the structure of stars \cite{Chandrasekhar,Endem.1907,Walker}.
Problem \eqref{intro.original} with $p=2$ was also studied by Barenblatt in relation to combustion theory in a volume edited by Gelfand \cite{Gelfand}. 
For general $p$, problem \eqref{intro.original}   is  often known in the literature as the ``Gelfand problem" or a ``Gelfand-type problem". It  was studied by Garc\'ia-Azorero, Peral, and Puel in  \cite{GP2, GPP}; see also \cite{Cabre.Sanchon.2007,Jacobsen.Schmitt.2002,Sanchon.2007} and the references therein. 

The asymptotic study of $p$-Laplacian problems as $p\to\infty$  offers a qualitative and quantitative understanding of their solution sets for large $p$, see \cite{Charro.Parini, Charro.Parini2, Charro.Peral,Charro.Peral2, Fukagai, Juutinen-Lindqvist-Manfredi}. Additionally, they have been used in \cite{Grosjean} to obtain optimal bounds for the diameter of manifolds in terms of their curvature.

In \cite{Charro.Parini,Charro.Parini2,Charro.Peral, Charro.Peral2, Fukagai, Juutinen-Lindqvist-Manfredi}, the authors study limits of $p$-Laplacian equations with  power-type right-hand sides and combinations of these. In all these cases, the parameter $\lambda$ is allowed to vary with $p$ in order to get nontrivial limits of sequences $\{u_{\lambda,p}\}_p$ of solutions to the corresponding $p$-Laplacian problem; namely,
\[
\lambda_p^{1/p}\to\Lambda\quad \textnormal{and} \quad u_{\lambda_{p},p}\to u \quad \textnormal{as} \ p\to\infty.
\]
With an exponential right-hand side, the solution sets change more drastically as $p\to\infty$ and  more severe rescalings become necessary.
To take limits in \eqref{intro.original}, we consider  
\begin{equation}\label{main.problem}
\left\{
\begin{aligned}
-&\Delta_{p} u_{\lambda_{p},p}=\lambda_{p}\,e^{u_{\lambda_{p},p}}&&\textrm{in}\ \Omega\\
&u_{\lambda_{p},p}=0 &&\textnormal{on}\ \partial\Omega,
\end{aligned}
\right.
\end{equation}
with the rescaling
\begin{equation}\label{normalization}
\frac{\lambda_p^{1/p}}{p}\to\Lambda\quad\text{as}\ p\to\infty.
\end{equation} 
Under this normalization, we prove that any uniform limit 
\begin{equation}\label{normalization.u}
\frac{u_{\lambda_{p},p}}{p}
\to u\quad\text{as}\ p\to\infty
\end{equation}
 is   a viscosity solution of the  limit problem
\begin{equation}\label{limite1}
\left\{
\begin{aligned}
&\min\left\{|\nabla{}u|-\Lambda\,e^{u},
-\Delta_{\infty}u\right\}=0&&\textrm{in}\ \Om,\\
&u=0 &&\text{on}\ \partial\Omega.
\end{aligned}
\right.
\end{equation}

It is worth noting that in 
\cite{Mihailescu-StancuDumitru-Varga},
the authors consider problem \eqref{intro.original} for  $\lambda$ independent of $p$ (i.e., without rescaling \eqref{normalization}) and obtain that  the corresponding solutions $u_p$ converge uniformly as $p\to\infty$ to the distance function to the  boundary of the domain. As the authors acknowledge in their paper, this result is not unexpected since for each  nonnegative function $f \in L^\infty(\Omega) \setminus \{0\}$, the sequence of unique solutions of 
\[
\left\{
\begin{aligned}
-&\Delta_{p}v_p=f(x)&&\text{in}\ \Omega\\
&v_p=0&&\text{on}\ \partial\Omega,
\end{aligned}
\right.
\]
converges uniformly in $\Omega$ to  the distance function to the boundary of the domain; see \cite{Bha-DiBe-Man,Kawohl90}.


In this paper, we prove passage to the limit of the sequence of minimal solutions of problem \eqref{main.problem} under the rescaling \eqref{normalization}, \eqref{normalization.u}. 
Furthermore, we show that the resulting limit is a minimal solution of \eqref{limite1}. Note that the fact that the limit solution is minimal is nontrivial; in principle, they could differ. To prove this, we use a  comparison principle for ``small solutions" of problem \eqref{limite1}, which we prove in  Section \ref{comparison}. As it turns out,  minimal solutions to problem \eqref{limite1} are ``small" in the sense of this comparison principle.
To the best of our knowledge, no corresponding comparison and uniqueness results for small solutions were known in the literature for $p<\infty$.

In Section \ref{section.explicit}, we find a second solution to the limit problem \eqref{limite1} under certain geometric assumptions on the domain $\Omega$. Furthermore, we show that both solutions lie on an explicit curve of solutions. Some examples of domains satisfying the geometric condition are  the ball, the annulus, and the stadium (convex hull of two balls of the same radius); a square or an ellipse does not verify the condition. 
We conjecture that this second solution is a limit of appropriately rescaled mountain-pass solutions of \eqref{main.problem}. 

The paper is organized as follows.  
In Section  \ref{sect.prelim},  we provide some necessary preliminaries,
and  Section \ref{section.limit.problem} formally introduces the limit problem. We have chosen to introduce the limit problem before proving any convergence results  to streamline the presentation. 
In Section \ref{comparison}, we prove the comparison principle for small solutions of the limit equation \eqref{limite1}.  
 Section \ref{sect.non.existence} concerns  non-existence of  solutions to \eqref{limite1} for large values of $\Lambda$.
In Section \ref{existence.minimal.limit.problem}, we find a branch of minimal solutions to  \eqref{limite1} up to a maximal $\Lambda$.
 Section \ref{p.minimal.limits} discusses  uniform convergence as $p\to\infty$ of  $p$-minimal solutions  to minimal  solutions of  \eqref{limite1}.
Finally, in Section \ref{section.explicit}, we show the multiplicity result  and  exhibit a curve of explicit solutions under a  geometric condition on the domain.


\section{Preliminaries}\label{sect.prelim}  

In this section, we state some necessary preliminaries and notation. First, let us recall that  weak solutions  of
 problem \eqref{intro.original} are also viscosity solutions.
The proof, which we omit
here, follows  
\cite[Lemma 1.8]{Juutinen-Lindqvist-Manfredi}; see also
\cite{Bha-DiBe-Man}.

\begin{lemma}
If $u$ is a continuous weak solution of \eqref{intro.original},
then
it is a viscosity solution of the same problem, rewritten as
\[
\left\{\begin{array}{l}
\displaystyle{F_{p}(\nabla{u},D^2u)=\lambda\,e^{u}\quad\textrm{in}\ \Om}\\
\displaystyle{u=0 \quad\textrm{on}\ \partial\Om,}
\end{array}\right.
\]
where
\begin{equation}\label{main.problem.2}
F_p(\xi,X)=-|\xi|^{p-2}\cdot\textnormal{trace}\left(\Big(I+(p-2)\frac{\xi\otimes\xi}{|\xi|^2}\Big)X\right).
\end{equation}
\end{lemma}

The divergence form of the $p$-Laplacian, i.e., $\textnormal{div}(|\nabla u|^{p-2}\nabla u)$, is better suited for variational techniques, while the expanded form \eqref{main.problem.2} is preferable in the viscosity framework.
In the sequel, we will always consider the most suitable form without  further mention.

%
%
%
%
%
%
%


In  \cite{Kawohl90}  the   problem 
\[
\left\{
\begin{aligned}
-&\Delta_{p}v_p=1&&\text{in}\ \Omega\\
&v_p\in W_{0}^{1,p}(\Omega)
\end{aligned}
\right.
\]
is studied in connection with torsional creep problems when $\Omega$ is a general bounded domain. Since we are interested in the case $p\to\infty$, we can assume  $p>n$ without loss of generality. 
Then every function in  $v_p\in W_0^{1,p}(\Omega)$ can be considered continuous in $\overline\Omega$ and  0 on the boundary in the classical sense.
The existence result we will need below is the following. We refer the interested reader to \cite{Kawohl90} and \cite[Theorem 3.11 and Remark 4.23]{JuutinenTESIS} for the proof.

\begin{prop}\label{yllegokawohlymandoaparar}
Let $\Omega$ be a bounded domain and  $n<p<\infty$. Then,
there exists a unique solution  $v_{p}\in W_{0}^{1,p}(\Omega)\cap{C}(\overline\Omega)$
of the $p$-torsion problem
\begin{equation}\label{aux.kawohl2}
\left\{
\begin{aligned}
-&\Delta_{p}v_{p}=1&&\text{in}\ \Omega\\
&v_{p}=0&&\text{on}\ \partial\Omega,
\end{aligned}
\right.
\end{equation}
and $v_{p}$ converge uniformly as $p\to\infty$ to the  unique viscosity solution to
\[
\left\{
\begin{aligned}
&\min\{|\nabla v|-1,-\Delta_\infty v\}=0&&\text{in}\ \Omega,\\
&v=0&&\text{on}\ \partial\Omega.
\end{aligned}
\right.
\]
Moreover, $v(x)=\textnormal{dist}(x,\partial\Omega)$.
\end{prop}

The uniqueness of the solution in Proposition \ref{yllegokawohlymandoaparar} follows from the following comparison principle.

\begin{lemma}\label{comparison.f(x)}
Let  $f:\Omega\to\mathbb{R}$ be a continuous, bounded, and positive function.
Suppose that $u,v:\overline\Omega\to\mathbb{R}$ are bounded, $u$ is upper semicontinuous and $v$ is lower semicontinuous in $\overline\Omega$. If $u$ and $v$ are, respectively, a viscosity sub- and supersolution of 
\[
\min\{|\nabla w|-f(x),-\Delta_\infty w\}=0\quad\text{in}\ \Omega,
\]
and $u \leq v$ on $\partial\Omega$, then $u \leq v$ in $\Omega$.
\end{lemma}

We refer the interested reader for instance to 
 \cite[Theorem 4.18 and Remark 4.23]{JuutinenTESIS} and also  \cite[Theorem 2.1]{Jensen93} (for the proof of  \cite[Theorem 4.18]{JuutinenTESIS}, notice that every $\infty$-superharmonic function is  Lipschitz continuous, see \cite{lindqvist-manfredi2}).



We  also need some facts about first eigenvalues and eigenfunctions of the $p$-Laplacian. Let us recall that the first eigenvalue $\lambda_1(p;\Omega)$ is characterized by the nonlinear Rayleigh quotient
\[
\lambda_1(p;\Omega)=\inf_{\phi\in W_0^{1,p}(\Omega)}\frac{\int_{\Omega}|\nabla \phi|^p\,dx}{\int_{\Omega}| \phi|^p\,dx}.
\]

In  \cite{lindqvist90} (see also \cite{lindqvist92}), it is proved that the first eigenvalue of the $p$-Laplacian is simple (that is,  the first eigenfunction is unique up to multiplication by constants) when  $\Omega$ is a  bounded domain; see also \cite{Anane,GP, IP} and the references in \cite{lindqvist90}. Moreover,  it is also proved in  \cite{lindqvist90} that in a  bounded domain,  only the first eigenfunction is positive and that the first eigenvalue is isolated  (there exists $\epsilon>0$ such that there are no eigenvalues in $(\lambda_1,\lambda_1+\epsilon]$).

\begin{prop}[\text{\cite{lindqvist90}}]
Let $\Omega$ be a bounded domain and  $n<p<\infty$. Then, there exists a  solution $\psi_{p}\in W_{0}^{1,p}(\Omega)\cap{C}(\overline\Omega)$ of
\begin{equation*}
\left\{
\begin{aligned}
-&\Delta_{p}\psi_{p}=\lambda_1(p;\Omega)\,|\psi_{p}|^{p-2}\psi_{p}&&\text{in}\ \Omega\\
&\psi_{p}=0 &&\text{on}\ \partial\Omega.
\end{aligned}
\right.
\end{equation*}
Moreover, $\lambda_1(p;\Omega)$ is simple and isolated.
\end{prop}


Lastly, we recall the behavior as $p\to\infty$ of the first eigenvalue of the $p$-Laplacian, see \cite{Juutinen-Lindqvist-Manfredi} for the proof.

\begin{lemma}\label{lema.autovalores}
$\displaystyle\lim_{p\to\infty}\lambda_{1}(p,\Omega)^{1/p}=\Lambda_1(\Omega)=
\|\textnormal{dist}(\cdot,\partial\Omega)\|_\infty^{-1}.$
\end{lemma}

 We denote the first $\infty$-eigenvalue by $\Lambda_1(\Omega)$, see \cite{Juutinen-Lindqvist-Manfredi}.

\medskip


\section{The limit problem}\label{section.limit.problem}

In the present section, we characterize uniform limits of appropriate rescalings of solutions of   \eqref{main.problem} as solutions of a PDE.  See  \cite{Charro.Parini, Charro.Parini2, Charro.Peral,Charro.Peral2, Fukagai, Juutinen-Lindqvist-Manfredi} for related results. 

\begin{prop}\label{proposition.limit.eq}
Consider  a sequence $\{(\lambda_p,u_{\lambda_p,p})\}_p$  of solutions of \eqref{main.problem} and assume
\[
\lim_{p\to\infty}
\frac{\lambda_p^{1/p}}{p}=\Lambda.
\]
Then, any uniform limit
\[
u_\Lambda=\lim_{p\to\infty}\frac{u_{\lambda_p,p}}{p}
\]
is a viscosity solution of the problem
\begin{equation}\label{eq:limite.1}
\left\{
\begin{aligned}
&\min\big\{|\nabla{}u|-\Lambda\,e^{u},
-\Delta_{\infty}u\big\}=0&&\textrm{in}\ \Om,\\
&u=0&& \textrm{on}\ \partial\Omega.
\end{aligned}
\right.
\end{equation}

\end{prop}

\begin{proof}
Consider a point $x_0\in\Om$ and a function
$\phi\in{C}^{2}(\Om)$ such that $u_\Lambda-\phi$ has a
strict local minimum at $x_0$. As $u_{\Lambda}$ is the uniform
limit of $u_{\lambda_{p},p}/p$, there exists a sequence of
points $x_p\rightarrow{}x_0$ such that $u_{\lambda_{p},p}-p\,\phi$ attains a local minimum at 
 $x_{p}$ for each
$p$. As $u_{\lambda_{p},p}$ is a continuous weak solution  of \eqref{main.problem}, it is also a viscosity solution and a supersolution. Then, we get
\[
\begin{split}
-(p-2)\,p^{p-1}\,|\nabla{}\phi(x_{p})|^{p-4}\,\bigg\{\frac{|\nabla{}\phi(x_{p})|^{2}}{p-2}\Delta{}\phi(x_{p})+\,&\langle{}D^2\phi(x_{p})\nabla{}\phi(x_{p}),\nabla{}\phi(x_{p})\rangle\bigg\}\\
&=-p^{p-1}\Delta_{p}\phi(x_{p})\geq \lambda_{p}\,e^{u_{\lambda,p}(x_{p})}.
\end{split}
\]
Rearranging terms, we obtain
\[
-(p-2)\left[\frac{|\nabla{}\phi(x_{p})|}{\left(\frac{\lambda_{p}}{p^{p-1}}\,e^{u_{\lambda,p}(x_p)}\right)^{\frac{1}{p-4}}}\right]^{p-4}
\bigg\{\frac{|\nabla{}\phi(x_{p})|^{2}}{p-2}\Delta{}\phi(x_{p})
+\langle{}D^2\phi(x_{p})\nabla{}\phi(x_{p}),\nabla{}\phi(x_{p})\rangle\bigg\}\geq1.
\]
If we suppose that
$
|\nabla\phi(x_0)|<\Lambda e^{u_{\Lambda}(x_0)}
$ we obtain a
contradiction letting $p\to\infty$ in the previous
inequality. Thus, it must be
\begin{equation}\label{eq1}
|\nabla\phi(x_0)|-\Lambda e^{u_{\Lambda}(x_0)}\geq0.
\end{equation}
We also have that
\begin{equation}\label{eq2}
-\Delta_{\infty}\phi(x_0)=-\langle{}D^2\phi(x_0)\nabla{}\phi(x_0),\nabla{}\phi(x_0)\rangle\geq0,
\end{equation}
because we would get a contradiction otherwise. Therefore, we can put together \eqref{eq1} and \eqref{eq2} writing
\begin{equation*}
\min\big\{|\nabla\phi(x_0)|-\Lambda e^{u_{\Lambda}(x_0)},
-\Delta_{\infty}\phi(x_0)\big\}\geq0,
\end{equation*}
and conclude that $u_{\Lambda}$ is a viscosity supersolution
of \eqref{eq:limite.1}.

It remains to show that $u_{\Lambda}$ is a viscosity
subsolution of the limit equation \eqref{eq:limite.1}. More precisely, we
have to show that, for each $x_0\in\Om$ and
$\phi\in{C}^2(\Om)$ such that $u_{\Lambda}-\phi$ attains a
strict local maximum at $x_0$ (note that $x_0$ and $\phi$ are not
the same than before) we have
\[
\min\left\{|\nabla\phi(x_0)|-\Lambda e^{u_{\Lambda}(x_0)},
-\Delta_{\infty}\phi(x_0)\right\}\leq0.
\]
We can suppose that
\[
|\nabla\phi(x_0)|>\Lambda e^{u_{\Lambda}(x_0)},
\]
since  we are done  otherwise. Again, the uniform
convergence of $u_{\lambda,p}/p$ to $u_{\Lambda}$ provides 
a sequence of points $x_{p}\rightarrow{}x_0$ which are local maxima of
$u_{\lambda,p}-p\,\phi$. Recalling the definition of viscosity subsolution we
have
\[
\begin{split}
-(p-2)\left[\frac{|\nabla{}\phi(x_{p})|}{\left(
\frac{\lambda_p}{p^{p-1}}  e^{{u_{\lambda,p}(x_{p})}}\right)^{\frac{1}{p-4}}}\right]^{p-4}
\bigg\{\frac{|\nabla{}\phi(x_{p})|^{2}}{p-2}\Delta{}\phi(x_{p})
+\langle{}D^2\phi(x_{p})\nabla{}\phi(x_{p}),\nabla{}\phi(x_{p})\rangle\bigg\}\leq1,
\end{split}
\]
for each  $p$. Letting $p\rightarrow\infty$, we find
$-\Delta_{\infty}\phi(x_0)\leq0$, or else we get a
contradiction.
\end{proof}

In the previous argument, the fact that $e^{u_\Lambda(x_0)}$ is strictly positive independently of the value of $u_\Lambda(x_0)$ makes a difference with the case with a  power-type right-hand side (see  \cite{Charro.Parini, Charro.Parini2, Charro.Peral,Charro.Peral2, Fukagai, Juutinen-Lindqvist-Manfredi}), where one needs to make sure that $u_\Lambda>0$ in $\Omega$. Furthermore, in the power-type right-hand side case, one can consider sign-changing solutions, see
\cite{Charro.Parini,Juutinen-Lindqvist} and get a more involved limit equation that takes into account sign changes. In the next result, we show that all solutions to 
 the limit problem \eqref{eq:limite.1} are positive. Moreover, we show that solutions cannot be arbitrarily small for every given $\Lambda$ and must grow (at least) linearly from the boundary.

\begin{prop}\label{noexist.small}
Let $\Omega\subset\R^n$ be a bounded domain and $\Lambda>0$. Then, every solution $u_\Lambda$ of \eqref{eq:limite.1} verifies
\[
u_{\Lambda}\geq \Lambda\,\textnormal{dist}(\cdot,\partial\Omega)\quad\text{in}\ \Omega.
\]
In particular, every solution of \eqref{eq:limite.1} is strictly positive and satisfies the estimate 
\begin{equation*}
\|u_{\Lambda}\|_{L^\infty(\Omega)}\geq\Lambda\Lambda_1(\Omega)^{-1}.
\end{equation*}
\end{prop}


\begin{proof}
 Let  $u_{\Lambda}$ be a  solution of  \eqref{eq:limite.1}. Then, $u_{\Lambda}\geq0$ in $\Omega$ by Lemma \ref{comparison.f(x)}. Let us show that 
\[
\min\{|\nabla u_{\Lambda}|-\Lambda,-\Delta_\infty u_{\Lambda}\}\geq0\quad\text{in}\ \Omega
\]
in the viscosity sense.
To see this, consider $x_0\in\Omega$ and
$\phi\in{C}^2$ such that $u_\Lambda-\phi$ has a minimum at
$x_0$. Since $ u_\Lambda(x)$ is a solution of  \eqref{eq:limite.1}, we have
\[
\min\left\{|\nabla{}\phi(x_0)|-\Lambda\,e^{ u_{\Lambda}(x_0)},
-\Delta_{\infty}\phi(x_0)\right\}\geq0\quad\text{in}\ \Omega.
\]
We deduce $-\Delta_{\infty}\phi(x_0)\geq0$ and
$|\nabla{}\phi(x_0)|\geq\Lambda\,e^{ u_{\Lambda}(x_0)}\geq\Lambda$ and we get
\[
\min\left\{|\nabla{}\phi(x_0)|-\Lambda,
-\Delta_{\infty}\phi(x_0)\right\}\geq0\quad\text{in}\ \Omega
\]
as desired.

On the other hand,  $v_\Lambda(x)=\Lambda\,\textnormal{dist}(x,\partial\Omega)$ is the unique viscosity solution of 
\[
\min\{|\nabla v_\Lambda|-\Lambda,-\Delta_\infty v_{\Lambda}\}=0\quad\text{in}\ \Omega.
\]
Then,  one gets  $u_{\Lambda}\geq v_\Lambda=\Lambda\,\textnormal{dist}(\cdot,\partial\Omega)$ by comparison, see 
Lemma \ref{comparison.f(x)}.
\end{proof}


\medskip


\section{Comparison for small solutions of the limit problem}\label{comparison}

In this section, we prove a comparison principle for small solutions of the limit equation \eqref{limite1}. This result is interesting for two main  reasons. Firstly, equation \eqref{limite1} is not proper in the terminology of \cite{CIL}, a basic requirement for comparison.  Secondly, based on the multiplicity results for the $p$-Laplacian equation \eqref{main.problem}, see  \cite{GP2, GPP}, one cannot expect comparison to hold in general. The key idea is a change of variables that allows us to obtain a proper equation for solutions with $\|u\|_\infty<1$. Remarkably, minimal solutions of \eqref{limite1} verify this condition (see Section \ref{existence.minimal.limit.problem} below), and we can conclude  they are the only ones with $\|u\|_\infty<1$. The change of variables  we use here is the same that was used to prove comparison for the limit problem with concave right-hand side in \cite{Charro.Peral}.

We prove a more general result with a ``right-hand" side $f(u)$ that satisfies a hypothesis reminiscent of the celebrated Brezis-Oswald condition, see \cite{Brezis-Oswald} and Remark \ref{Brezis.Oswald.condition} below. 
\begin{theorem}\label{theorem.ppio.comparacion}
Let   $f:\mathbb{R}\to\mathbb{R}$ be a continuous  function for which there exist $c\in(0,\infty]$ and $q\in(0,1)$ such that 
\begin{equation}\label{condicion.crecimiento.f.inflap}
\frac{f(t)}{t^q}\quad\text{is positive and non-increasing for all}\ t\in(0,c).
\end{equation}
Let $\Omega\subset\R^n$ be a bounded domain and let 
$u,v\in C(\overline\Omega)$ with $\max\{\|u\|_\infty,\|v\|_\infty\}<c$ be, respectively, a positive viscosity sub- and supersolution of
\begin{equation}\label{eq:ppio.comp}
\min\left\{|\nabla{}w|-f(w),
-\Delta_{\infty}w\right\}=0\quad\textrm{in}\ \Omega.
\end{equation}
 Then, whenever $u\leq{}v$ on
$\partial\Omega$, we have $u\leq{}v$ in $\overline{\Omega}$.
\end{theorem}

\begin{remark}\label{Brezis.Oswald.condition}
It is possible to prove a comparison principle for equation \eqref{eq:ppio.comp} under the  Brezis-Oswald \cite{Brezis-Oswald} condition
\[
\frac{f(t)}{t}\quad\text{is decreasing for all}\  t > 0.
\]
Under this condition, the power-type change of variables used in \cite{Charro.Peral} and in the proof of Theorem ~\ref{theorem.ppio.comparacion} no longer applies. Instead, we need a logarithmic change of variables, similarly to the comparison principle for the eigenvalue problem for the infinity Laplacian in \cite{Juutinen-Lindqvist-Manfredi}. However, a viscosity comparison principle obtained through a logarithmic change of variables requires that either the sub- or the supersolution are strictly positive in $\overline\Omega$ and does not allow us to conclude uniqueness of solutions for the Dirichlet problem with homogeneous boundary data, which our result does.
\end{remark}


As a consequence of Theorem \ref{theorem.ppio.comparacion}, we have uniqueness of ``small" solutions of problem \eqref{eq:limite.1}.
\begin{corollary}\label{corolario.unicidad.limit.concavo}
Let $\Omega\subset\R^n$ be a bounded domain.
For every $\Lambda>0$, the problem 
\begin{equation}\label{corolry.eq:ppio.comp}
\left\{
\begin{aligned}
&\min\left\{|\nabla{}u|-\Lambda\,e^{u},
-\Delta_{\infty}u\right\}=0&&\textrm{in}\ \Omega,\\
&u=0&& \text{on}\ \partial\Omega,
\end{aligned}
\right.
\end{equation}
has at most one viscosity solution with $\|u\|_\infty< 1$.
\end{corollary}

\begin{proof}[Proof of Corollary \ref{corolario.unicidad.limit.concavo}]
 Suppose for the sake of contradiction that there are two viscosity solutions, $u,v$ of \eqref{corolry.eq:ppio.comp} with $\max\{\|u\|_\infty,\|v\|_\infty\}<1.$  Notice that  both $u$ and $v$ are  strictly positive in $\Omega$ by Proposition \ref{noexist.small}.
 In this case we have $f(t)=\Lambda\,e^t$ and \eqref{condicion.crecimiento.f.inflap} is satisfied with $c=q$ for every $q\in(0,1)$.
 Then, we can choose $q\in(0,1)$ such that 
 $ \max\{\|u\|_\infty,\|v\|_\infty\}<q<1,$
 and all the hypotheses of Theorem \ref{theorem.ppio.comparacion} are satisfied. Because $u=v$ on $\partial \Omega,$  we conclude $u\equiv v$.
\end{proof}

In the next lemma we apply a change of variables to equation \eqref{eq:ppio.comp}.

\begin{lemma}\label{lemma.cambio.variables.inflap}
Let $q\in(0,1)$ and let $v$ be a positive viscosity supersolution (respectively, subsolution) of \eqref{eq:ppio.comp} in $\Omega$. Then,
$\tilde{v}(x)=v^{1-q}(x)$ is a viscosity supersolution (subsolution) of
\begin{equation}\label{eq:15}
\min\left\{
|\nabla{}\tilde{w}(x)|
-
 (1-q)\,\frac{f\left({\tilde{w}(x)^{\frac{1}{1-q}}}\right)}{\tilde{w}(x)^{\frac{q}{1-q}}},
-\Delta_{\infty}\tilde{w}(x)-\frac{q}{1-q}\frac{|\nabla{}\tilde{w}(x)|^4}{\tilde{w}(x)}\right\}=0
\end{equation}
in  every subdomain $U$ compactly contained in $\Omega$.
\end{lemma}

\begin{proof}
Let $\tilde{\phi}\in C^2(\Omega)$ touch $\tilde{v}$ from below at $x_0\in\Omega$. If we define
$\phi(x)=\tilde{\phi}(x)^{\frac{1}{1-q}}$, then  $\phi$ touches 
$v$ from below at $x_0$. Note that  $\phi(x)$ is $ C^2$
in a neighborhood of $x_0$, since $v>0$ in $\Omega$ implies
$\tilde{\phi}(x)>0$ around $x_0$. Then
\begin{eqnarray*}
&&\nabla\phi(x_0)=\frac{1}{1-q}\,\tilde{\phi}(x_0)^{\frac{q}{1-q}}\,\nabla\tilde{\phi}(x_0),\\
&&D^2\phi(x_0)=\frac{1}{1-q}\,\tilde{\phi}(x_0)^{\frac{q}{1-q}}D^2\tilde{\phi}(x_0)+\frac{q}{(1-q)^2}\,\tilde{\phi}(x_0)^{\frac{2q-1}{1-q}}\,\nabla\tilde{\phi}(x_0)\otimes\nabla\tilde{\phi}(x_0).
\end{eqnarray*}
Because $v$ is a  viscosity supersolution of \eqref{eq:ppio.comp} and $\phi(x_0)=v(x_0)>0$, we have
\begin{eqnarray*}
0&\leq&\min\left\{|\nabla{}\phi(x_0)|- f\big({\phi(x_0)}\big),
-\langle{}D^2\phi(x_0)\nabla\phi(x_0),\nabla\phi(x_0)\rangle\right\}\\
&=&\min\Bigg\{
\frac{1}{1-q}\,\tilde{\phi}(x_0)^{\frac{q}{1-q}}\Bigg(|\nabla\tilde{\phi}(x_0)|- (1-q)\,\frac{f\Big({\tilde{\phi}(x_0)^{\frac{1}{1-q}}}\Big)}{\tilde{\phi}(x_0)^{\frac{q}{1-q}}}\Bigg),\\
&&\hspace{30pt}-\Big(\frac{1}{1-q}\,\tilde{\phi}(x_0)^{\frac{q}{1-q}}\Big)^3\,\left(\Delta_\infty\tilde{\phi}(x_0)+\frac{q}{1-q}\,\frac{|\nabla\tilde{\phi}(x_0)|^4}{\tilde{\phi}(x_0)}\right)\Bigg\}.
\end{eqnarray*}
Therefore,
\[
\min\left\{
|\nabla\tilde{\phi}(x_0)|- (1-q)\,\frac{f\Big({\tilde{\phi}(x_0)^{\frac{1}{1-q}}}\Big)}{\tilde{\phi}(x_0)^{\frac{q}{1-q}}},
-
\Delta_\infty\tilde{\phi}(x_0)-\frac{q}{1-q}\,\frac{|\nabla\tilde{\phi}(x_0)|^4}{\tilde{\phi}(x_0)}
\right\}
\geq
0,
\]
that is,  $\tilde{v}$ is a viscosity supersolution  of
\eqref{eq:15}. The subsolution case is analogous.
\end{proof}

Equation \eqref{eq:15} is given by the functional
\[
\begin{aligned}
\mathcal{F}:\R^+\times\R^n\times{}\mathcal{S}^n&\longrightarrow&\R&\\
(t,p,X)&\longrightarrow&&
\hspace{-9pt}\min\left\{|p|-
 (1-q)
f\Big(t^{\frac{1}{1-q}}\Big)t^{-\frac{q}{1-q}},
-\langle{}Xp,p\rangle-\frac{q}{1-q}\,\frac{|p|^4}{t}\right\},
\end{aligned}
\]
which is  degenerate elliptic and non-decreasing in $t$ for $0<t<c^{1-q}$ by hypothesis \eqref{condicion.crecimiento.f.inflap}. Under these conditions, it is well-known (see  \cite[Section 5.C]{CIL}) that it is possible to establish a comparison  principle when  the supersolution or the subsolution are strict. In the next lemma we show that  we can find a perturbation of the  supersolution that is a strict supersolution, see  \cite{Charro.Peral,JuutinenTESIS,Juutinen-Lindqvist-Manfredi} for related constructions.

\begin{lemma}\label{lema.strict.supersolution.inflap}
Consider a subdomain $U$ compactly contained in $\Omega$, and $q\in(0,1),$ $c>0$ as in \eqref{condicion.crecimiento.f.inflap}. Let $\tilde{v}>0$ with $\|\tilde{v}\|_\infty<c^{1-q}$  be a viscosity supersolution  of
\eqref{eq:15} in $U$.
Define
\begin{equation}\label{cap3:eq:17}
\tilde{v}_\epsilon(x)=(1+\epsilon)\big(\tilde{v}(x)+\epsilon\big).
\end{equation}
Then,  $\tilde{v}_\epsilon\to \tilde{v}$ uniformly in  $\overline{U}$ as $\epsilon\to0$,
and for every $\epsilon>0$ small enough,
there exists a positive constant $C=C(\epsilon,q,\|\tilde{v}\|_\infty)$ such that
\begin{equation}\label{lemma.supersol.strict.estatement}
\min\left\{
|\nabla\tilde{v}_{\epsilon}(x)|  - (1-q)\,\frac{f\left({\tilde{v}_\epsilon(x)^{\frac{1}{1-q}}}\right)}{\tilde{v}_\epsilon(x)^{\frac{q}{1-q}}}
,-\Delta_\infty\tilde{v}_{\epsilon}(x)-\frac{q}{1-q}\,\frac{|\nabla{}\tilde{v}_{\epsilon}(x)|^4}{\tilde{v}_\epsilon(x)}\right\}
\geq
C>0\quad\textrm{in}\ U,
\end{equation}
in the viscosity sense, that is, $\tilde{v}_\epsilon$
is a strict viscosity supersolution of \eqref{eq:15} in $U$ with $\|\tilde{v}_\epsilon\|_\infty<c^{1-q}$.
 \end{lemma}

\begin{proof}
Let $\tilde{\phi}_{\epsilon}\in C^2$ touch $\tilde{v}_{\epsilon}(x)$ from below at  $x_0\in U$. Define
\[
\tilde{\phi}(x)=\frac{1}{1+\epsilon}\,\tilde{\phi}_{\epsilon}(x)-\epsilon,
\]
which clearly touches  $\tilde{v}(x)$ from below at
$x_0$. Then,
\begin{equation}\label{cap3:eq:19bis}
\nabla{\tilde{\phi}}(x_0)=(1+\epsilon)^{-1}\,\nabla{}\tilde{\phi}_{\epsilon}(x_0)\quad\text{and}\quad
D^2\tilde{\phi}(x_{0})=(1+\epsilon)^{-1}\,D^2\tilde{\phi}_{\epsilon}(x_{0}).
\end{equation}
Since $\tilde{v}(x)$ is a viscosity supersolution of \eqref{eq:15} in
$U$, we deduce 
\begin{equation}\label{eq:20}
|\nabla{}\tilde{\phi}(x_{0})|
\geq
 (1-q)\,\frac{f\left({\tilde{v}(x_0)^{\frac{1}{1-q}}}\right)}{\tilde{v}(x_0)^{\frac{q}{1-q}}},
\end{equation}
 and
\begin{equation}\label{eq:21}
-\big\langle{}D^2\tilde{\phi}(x_{0})\nabla{}\tilde{\phi}(x_{0}),\nabla{}\tilde{\phi}(x_{0})\big\rangle-\frac{q}{1-q}\,\frac{|\nabla{}\tilde{\phi}(x_{0})|^4}{\tilde{v}(x_{0})}\geq0.
\end{equation}
In the sequel we assume $\epsilon$ small enough so that $\|\tilde{v}\|_\infty<\|\tilde{v}_\epsilon\|_\infty=(1+\epsilon)\big(\|\tilde{v}\|_\infty+\epsilon\big)<c^{1-q}.$
Then, from \eqref{condicion.crecimiento.f.inflap},  \eqref{cap3:eq:17}, \eqref{cap3:eq:19bis} and \eqref{eq:20}, we obtain
\begin{equation}\label{eq:22}
\begin{split}
|\nabla\tilde{\phi}_{\epsilon}(x_{0})| & - (1-q)\,\frac{f\left({\tilde{v}_\epsilon(x_0)^{\frac{1}{1-q}}}\right)}{\tilde{v}_\epsilon(x_0)^{\frac{q}{1-q}}}
\\
 &\geq{}
\epsilon\,
 (1-q)\,\frac{f\left({\tilde{v}(x_0)^{\frac{1}{1-q}}}\right)}{\tilde{v}(x_0)^{\frac{q}{1-q}}}
+
 (1-q)\,
 \left(
 \frac{f\left({\tilde{v}(x_0)^{\frac{1}{1-q}}}\right)}{\tilde{v}(x_0)^{\frac{q}{1-q}}}
 -
 \frac{f\left({\tilde{v}_\epsilon(x_0)^{\frac{1}{1-q}}}\right)}{\tilde{v}_\epsilon(x_0)^{\frac{q}{1-q}}}
 \right)
 \\
 &\geq{}
\epsilon\,
 (1-q)\,\frac{f\left(\|{\tilde{v}\|_\infty^{\frac{1}{1-q}}}\right)}{\|\tilde{v}\|_\infty^{\frac{q}{1-q}}}.
 \end{split}
\end{equation}
Similarly, from \eqref{condicion.crecimiento.f.inflap},  \eqref{cap3:eq:17}, \eqref{cap3:eq:19bis}, \eqref{eq:20}, and \eqref{eq:21} we arrive at
\begin{equation}\label{eq:24}
\begin{split}
-&\big\langle{}D^2\tilde{\phi}_{\epsilon}(x_{0})\nabla{}\tilde{\phi}_{\epsilon}(x_{0}),\nabla{}\tilde{\phi}_{\epsilon}(x_{0})\big\rangle-\frac{q}{1-q}\,\frac{|\nabla{}\tilde{\phi}_{\epsilon}(x_{0})|^4}{\tilde{v}_\epsilon(x_{0})}\\
&\geq{}
(1+\epsilon)^3\frac{q}{1-q}\left(\frac{1}{\tilde{v}(x_{0})}-\frac{1}{\tilde{v}(x_{0})+\epsilon}\right)
|\nabla{}\tilde{\phi}(x_{0})|^4\\
&\geq{}
\frac{\epsilon(1+\epsilon)^3q(1-q)^3
}{\|\tilde v\|_\infty\left(\|\tilde v\|_\infty+\epsilon\right)}
\left(
\frac{f\left({\tilde{v}(x_0)^{\frac{1}{1-q}}}\right)}{\tilde{v}(x_0)^{\frac{q}{1-q}}}
 \right)^4
\geq{}
\frac{\epsilon(1+\epsilon)^3q(1-q)^3
}{\|\tilde v\|_\infty\left(\|\tilde v\|_\infty+\epsilon\right)}
\left(
\frac{f\left({\|\tilde{v}\|_\infty^\frac{1}{1-q}}\right)}{\|\tilde{v}\|_\infty^\frac{q}{1-q}}
\right)^4.
\end{split}
\end{equation}
Finally, we get \eqref{lemma.supersol.strict.estatement}  from  \eqref{eq:22} and \eqref{eq:24} 
as desired, which concludes the proof.
\end{proof}

\begin{proof}[Proof of Theorem \ref{theorem.ppio.comparacion}]

Since $u-v\in C(\overline{\Omega})$ and $\overline{\Omega}$ is
compact, $u-v$ attains its maximum in  $\overline{\Omega}$. Suppose, for the sake of contradiction, that
$\max_{\overline{\Omega}}(u-v)>0$. Let
\[
\tilde{u}(x)=u(x)^{1-q},\qquad\tilde{v}(x)=v(x)^{1-q},
\]
and define $\tilde{v}_{\epsilon}(x)$
as in \eqref{cap3:eq:17}.
Notice that  $u-v\leq0$ on $\partial\Omega$ gives
\[
\tilde{u}-\tilde{v}_{\epsilon}=\tilde{u}-(1+{\epsilon})\,\tilde{v}-(1+{\epsilon})\;{\epsilon}<0
\quad\text{on}\ \partial\Omega.
\]
Moreover, by uniform convergence, we have
$\max_{\overline{\Omega}}(\tilde{u}-\tilde{v}_{\epsilon})>0$ for
${\epsilon}$ small enough. Therefore,  we can fix   $\epsilon>0$ small as in Lemma \ref{lema.strict.supersolution.inflap} for the rest of the proof and assume there exists  $U$ compactly contained in $\Omega$ that contains all maximum points of
$\tilde{u}-\tilde{v}_{\epsilon}$.
We have proved in Lemmas \ref{lemma.cambio.variables.inflap} and \ref{lema.strict.supersolution.inflap} that
$\tilde{u}$ and $\tilde{v}_{\epsilon}$ are, respectively, a viscosity subsolution and strict supersolution of \eqref{eq:15} in $U$.

For every $\tau>0$, let $(x_\tau,y_\tau)$ be a maximum point of 
 $\tilde{u}(x)-\tilde{v}_{{\epsilon}}(y)-\frac{\tau}{2}|x-y|^2$
in $\overline{\Omega}\times\overline{\Omega}$. By the compactness of
$\overline{\Omega}$, we can assume that $x_\tau\rightarrow\hat{x}$ as
 $\tau\rightarrow\infty$ for some $\hat{x}\in\overline{\Omega}$
(notice that also $y_\tau\rightarrow\hat{x}$).  Then, \cite[Proposition 3.7]{CIL} implies that $\hat{x}$ is a maximum point of  $\tilde{u}-\tilde{v}_{{\epsilon}}$ and, therefore, an interior point of   $U$. We also have that
\[
\lim_{\tau\rightarrow\infty}\left(\tilde{u}(x_\tau)-\tilde{v}_{{\epsilon}}(y_\tau)-\frac{\tau}{2}|x_\tau-y_\tau|^2\right)=\tilde{u}(\hat{x})-\tilde{v}_{{\epsilon}}(\hat{x})>0,
\]
and, consequently, both $x_\tau$ and
$y_\tau$ are interior points of $U$  for $\tau$ large enough and
\begin{equation}\label{eq:27}
\tilde{u}(x_\tau)-\tilde{v}_{{\epsilon}}(y_\tau)-\frac{\tau}{2}|x_\tau-y_\tau|^2>0.
\end{equation} 
The definition of viscosity solution and the maximum principle for semicontinuous functions, see \cite{CIL}, imply that there exist symmetric matrices 
 $X_\tau$, $Y_\tau$ with $X_\tau\leq{}Y_\tau$  such that
%
%
%
\[
\min
\Bigg\{
\tau\,|x_\tau-y_\tau| - (1-q)\,\frac{f\left({\tilde{u}(x_\tau)^{\frac{1}{1-q}}}\right)}{\tilde{u}(x_\tau)^{\frac{q}{1-q}}}
,
-\tau^2\langle{}X_\tau(x_\tau-y_\tau),(x_\tau-y_\tau)\rangle-\frac{q}{1-q}\,\frac{\tau^4|x_\tau-y_\tau|^4}{\tilde{u}(x_\tau)}
\Bigg\}
\leq
0,
\]
and
\[
\begin{split}
\min
\Bigg\{
\tau\,|x_\tau-y_\tau| & -  (1-q)\,\frac{f\left({\tilde{v}_\epsilon(y_\tau)^{\frac{1}{1-q}}}\right)}{\tilde{v}_\epsilon(y_\tau)^{\frac{q}{1-q}}},
\\
&-\tau^2\langle{}Y_\tau(x_\tau-y_\tau),(x_
\tau-y_\tau)\rangle-\frac{q}{1-q}\,\frac{\tau^4|x_\tau-y_\tau|^4}{\tilde{v}_{\epsilon}(y_\tau)}\Bigg\}
\geq C(\epsilon,q,\|\tilde v\|_\infty)>0.
\end{split}
\]
Subtracting both equations, we get
\begin{eqnarray}
\hspace{-20pt}
0
<
C(\epsilon,q,\|\tilde{v}\|_\infty)
\!\!\!\!&\leq&\!\!\!\!
\min\Bigg\{
\tau\,|x_\tau-y_\tau|  -  (1-q)\,\frac{f\left({\tilde{v}_\epsilon(y_\tau)^{\frac{1}{1-q}}}\right)}{\tilde{v}_\epsilon(y_\tau)^{\frac{q}{1-q}}},
\notag
\\
&&\hspace{25pt}-\tau^2\langle{}Y_\tau(x_\tau-y_\tau),(x_\tau-y_\tau)\rangle-\frac{q}{1-q}\,\frac{\tau^4|x_\tau-y_\tau|^4}{\tilde{v}_{\epsilon}(y_\tau)}\Bigg\}
\label{eq:31}
\\
&&\!\!\!\!-
\min
\Bigg\{
\tau\,|x_\tau-y_\tau| - (1-q)\,\frac{f\left({\tilde{u}(x_\tau)^{\frac{1}{1-q}}}\right)}{\tilde{u}(x_\tau)^{\frac{q}{1-q}}},
\notag
\\
&&\hspace{25pt}-\tau^2\langle{}X_\tau(x_\tau-y_\tau),(x_\tau-y_\tau)\rangle-\frac{q}{1-q}\,\frac{\tau^4|x_\tau-y_\tau|^4}{\tilde{u}(x_\tau)}
\Bigg\}.\label{eq:32}
\end{eqnarray}
We  consider four cases, depending on the values where the minima in \eqref{eq:31} and \eqref{eq:32} are attained.  In all cases we obtain a contradiction using that
$X_\tau\leq{}Y_\tau$ and $\tilde{v}_{\epsilon}(y_\tau)\leq{}\tilde{u}(x_\tau)$, which follows from ~\eqref{eq:27}.
\begin{enumerate}\itemsep3pt

 \item{}Both minima  are attained by the first terms and \eqref{condicion.crecimiento.f.inflap} implies a contradiction, i.e.,
 \[
 0
<
C(\epsilon,q,\|\tilde{v}\|_\infty)
\leq
 (1-q)\left(
 \frac{f\left({\tilde{u}(x_\tau)^{\frac{1}{1-q}}}\right)}{\tilde{u}(x_\tau)^{\frac{q}{1-q}}}
   -  \frac{f\left({\tilde{v}_\epsilon(y_\tau)^{\frac{1}{1-q}}}\right)}{\tilde{v}_\epsilon(y_\tau)^{\frac{q}{1-q}}}
   \right)
   \leq0.
\]

\item{}Both minima are attained by the second terms. Then,
\[
\begin{split}
0
&<
C(\epsilon,q,\|\tilde{v}\|_\infty)
\\
&\leq
-\tau^2\big\langle{}(Y_\tau-X_\tau)(x_\tau-y_\tau),(x_\tau-y_\tau)\big\rangle
+\frac{q}{1-q}\;\tau^4|x_\tau-y_\tau|^4
\left(
\frac{1}{\tilde{u}(x_\tau)}
-
\frac{1}{\tilde{v}_{\epsilon}(y_\tau)}
\right)\leq0,
\end{split}
\]
a contradiction.

 \item{}The minima in \eqref{eq:31} and \eqref{eq:32} are attained by the second  and first term, respectively. This case can be reduced to case (1) above and we again obtain a contradiction. Namely,
 \[
\begin{split}
0
<
C(\epsilon,q,\|\tilde{v}\|_\infty)
\leq
&-\tau^2\langle{}Y_\tau(x_\tau-y_\tau),(x_\tau-y_\tau)\rangle-\frac{q}{1-q}\,\frac{\tau^4|x_\tau-y_\tau|^4}{\tilde{v}_{\epsilon}(y_\tau)}
\\
&-\tau\,|x_\tau-y_\tau| + (1-q)\,\frac{f\left({\tilde{u}(x_\tau)^{\frac{1}{1-q}}}\right)}{\tilde{u}(x_\tau)^{\frac{q}{1-q}}}
\\
\leq
&\,
 (1-q)\left(
 \frac{f\left({\tilde{u}(x_\tau)^{\frac{1}{1-q}}}\right)}{\tilde{u}(x_\tau)^{\frac{q}{1-q}}}
   -  \frac{f\left({\tilde{v}_\epsilon(y_\tau)^{\frac{1}{1-q}}}\right)}{\tilde{v}_\epsilon(y_\tau)^{\frac{q}{1-q}}}
   \right)
   \leq0.
\end{split}
\]

 \item{}Finally, if the minima in \eqref{eq:31} and \eqref{eq:32} are respectively attained by the first  and second term, we obtain a contradiction as in case (2) above, i.e.,
 \[
\begin{split}
0
<
&\;
C(\epsilon,q,\|\tilde{v}\|_\infty)
\leq
\tau\,|x_\tau-y_\tau|  -  (1-q)\,\frac{f\left({\tilde{v}_\epsilon(y_\tau)^{\frac{1}{1-q}}}\right)}{\tilde{v}_\epsilon(y_\tau)^{\frac{q}{1-q}}}
\\
&\hspace{+74pt}+\tau^2\langle{}X_\tau(x_\tau-y_\tau),(x_\tau-y_\tau)\rangle
+\frac{q}{1-q}\,\frac{\tau^4|x_\tau-y_\tau|^4}{\tilde{u}(x_\tau)}\\
\leq
&
-\tau^2\big\langle{}(Y_\tau-X_\tau)(x_\tau-y_\tau),(x_\tau-y_\tau)\big\rangle
+\frac{q}{1-q}\;\tau^4|x_\tau-y_\tau|^4
\left(
\frac{1}{\tilde{u}(x_\tau)}
-
\frac{1}{\tilde{v}_{\epsilon}(y_\tau)}
\right)\leq0.
\end{split}
\]
\end{enumerate}

\medskip
\noindent{}Since all the alternatives lead to a contradiction, the proof is complete.
\end{proof}

%
%

\medskip


\section{Non-existence of  solutions with large $\Lambda$ for the limit problem}\label{sect.non.existence}

We show here that due to the structure of  the limit problem \eqref{eq:limite.1}, there exists a threshold
 $\Lambda_{\max}$ beyond which the problem has no solutions.

\begin{prop}\label{prop:NO.exist.beyond.Lambda.hat}
Let $\Omega\subset\R^n$ be a bounded domain. Problem \eqref{eq:limite.1} has no solutions for $\Lambda>\Lambda_{\max}$, where
\begin{equation}\label{defn.lambda.hat}
\Lambda_{\max}=e^{-1}\Lambda_{1}(\Omega),
\end{equation}
and $\displaystyle{\Lambda_1(\Omega)\,=\|\textnormal{dist}(\cdot,\partial\Omega)\|_\infty}^{-1}$ is the first $\infty$-eigenvalue, see \cite{Juutinen-Lindqvist-Manfredi}.
\end{prop}

\begin{proof}
Define $\mu=\Lambda_1(\Omega)+\epsilon$  with $\epsilon>0$. Suppose for contradiction that problem \eqref{eq:limite.1} has a solution $u_\Lambda$ for some
$\Lambda>e^{-1}\mu$. 

First we are going to use this $u_\Lambda$ to construct a supersolution to the eigenvalue  problem with parameter $\mu$. More precisely, we are going to show that
\begin{equation}\label{eq:CCNO:supersol}
\min\big\{|\nabla{}u_{\Lambda}|-\mu\, u_{\Lambda},
-\Delta_{\infty}u_{\Lambda}\big\}\geq0\quad\text{in}\ \Omega
\end{equation}
in the viscosity sense. To this aim, let $x_0\in\Omega$ and
$\phi\in{C}^2$ such that $u_\Lambda-\phi$ has a minimum in 
$x_0$. Since $ u_\Lambda(x)$ is a solution of problem \eqref{eq:limite.1} we have
\[
\min\left\{|\nabla{}\phi(x_0)|-\Lambda\,e^{ u_{\Lambda}(x_0)},
-\Delta_{\infty}\phi(x_0)\right\}\geq0\quad\text{in}\ \Omega.
\]
We deduce that $-\Delta_{\infty}\phi(x_0)\geq0$ and
$|\nabla{}\phi(x_0)|\geq\Lambda\,e^{ u_{\Lambda}(x_0)}$. Hence,
\[
|\nabla{}\phi(x_0)|-\mu\,u_\Lambda(x_0)\geq\Lambda\,e^{ u_{\Lambda}(x_0)}-\mu\,u_\Lambda(x_0).
\]
To deduce \eqref{eq:CCNO:supersol} it is enough to show that
\[
\min_{t\in\mathbb{R}}\Phi_\Lambda(t)\geq0\quad\text{where}\quad\Phi_\Lambda(t)=\Lambda\,e^{t}-\mu\,t.
\]
It is elementary to check that the function  $\Phi_{\Lambda}$ is convex and has a unique minimum point at $t_\text{min}=\log(\mu\,\Lambda^{-1})$. Notice that $\lim_{t\to\pm\infty}\Phi_{\Lambda}(t)=+\infty$, and hence $t_\text{min}$ is a global minimum. Then, it is easy to check that $\Lambda>e^{-1}\mu$ implies
$\Phi_{\Lambda}(t_\text{min})\geq0$.

\medskip

Next, we notice that any first $\infty$-eigenfunction is a subsolution of the eigenvalue problem with parameter $\mu$. So, let $v$ be a first $\infty-$eigenfunction, that is, a solution of 
\[
\left\{
\begin{aligned}
&\min\big\{|\nabla{} v|-\Lambda_1(\Omega)\,   v,
-\Delta_{\infty}  v\big\}=0&&\text{in}\ \Omega,\\
&  v>0 && \text{in}\ \Omega \\
&  v=0 && \text{on}\ \partial\Omega
\end{aligned}
\right.
\]
normalized in such a way that $\| v\|_\infty<e^{-1}$. Clearly, by definition of $\mu$, 
\[
\min\big\{|\nabla{} v|-\mu\,   v,
-\Delta_{\infty}  v\big\}\leq0\quad\text{in}\ \Omega.
\]

\medskip

Now, we have to show that $u_\Lambda$ and $v$ are ordered, namely, that  $0<v\leq u_\Lambda$ in $\Omega$. Indeed, using that $\| v\|_\infty<e^{-1}$ and $\Lambda_1(\Omega)<\mu<\Lambda e$, it is easy to see that
\[
\min\big\{|\nabla{} v |-\Lambda,
-\Delta_{\infty} v\big\}\leq0\quad\text{in}\ \Omega,
\]
and using that $e^{u_\Lambda(x)}\geq1$ in $\Omega$ one gets
\[
\min\big\{|\nabla{} u_\Lambda |-\Lambda,
-\Delta_{\infty} u_\Lambda\big\}\geq0\quad\text{in}\ \Omega.
\]
As $v=u_\Lambda=0$ on $\partial\Omega$,  we get $0<v\leq u_\Lambda$ by comparison, see 
Lemma \ref{comparison.f(x)}.

\medskip

So far, we have a subsolution $v$ and a supersolution  $u_\Lambda$ of the eigenvalue problem
\begin{equation}\label{cafeslaestrella}
\min\big\{|\nabla{} w|-\mu\,  w,
-\Delta_{\infty} w\big\}=0\quad\text{in}\ \Omega
\end{equation}
which verify $0<v \leq u_\Lambda$. Next we claim that it is possible to  construct  a solution of  \eqref{cafeslaestrella} iterating between $v$ and $u_\Lambda$.
The argument finishes noticing that we have constructed a positive $\infty-$eigenfunction associated to $\mu=\Lambda_1+\epsilon$, which is a contradiction with the fact that $\Lambda_1$ is isolated (see \cite[Theorem 8.1]{Juutinen-Lindqvist} and \cite[Theorem 3.1]{Juutinen-Lindqvist-Manfredi}). Since the argument above works for every  $\epsilon>0$, we conclude that there is no solution of  \eqref{eq:limite.1}  for $\Lambda>\Lambda_{\max}$.

\medskip

We conclude by proving the claim. First, define  $w_1(x)$, viscosity solution of
\[
\left\{
\begin{aligned}
&\min\big\{|\nabla{} w_1|-\mu\,  v,-\Delta_{\infty} w_1\big\}=0&&\text{in}\ \Omega\\
&w_1=0&&\text{on}\ \partial\Omega.
\end{aligned}
\right.
\]
To prove that such a $w_1$ exists, notice that $ v $ is a subsolution of the problem and that $ u_\Lambda$ is a supersolution, since, from \eqref{eq:CCNO:supersol} and $v\leq u_\Lambda$ we deduce
\[
\min\big\{|\nabla{} u_\Lambda|-\mu\,  v,-\Delta_{\infty} u_\Lambda\big\}\geq0.
\]
Then, we can apply the comparison principle as above and apply the Perron method  (\cite[Theorem 4.1]{CIL}), to get a unique  $w_1$ such that
\[
 v \leq w_1\leq u_\Lambda\quad \text{in}\ \Omega.
\]
Then, we define  $w_2$, the solution of
\[
\left\{
\begin{aligned}
&\min\big\{|\nabla{} w_2|-\mu\,  w_1,-\Delta_{\infty} w_2\big\}=0&&\text{in}\ \Omega\\
&w_2=0&&\text{on}\ \partial\Omega.
\end{aligned}
\right.
\]
In this case, $w_1$ is a subsolution and $ u_\Lambda$ is a supersolution, since
\[
\min\big\{|\nabla{} w_1|-\mu\,  v,-\Delta_{\infty} w_1\big\}=0\quad\Rightarrow\quad \min\big\{|\nabla{} w_1|-\mu\,  w_1,-\Delta_{\infty} w_1\big\}\leq0,
\]
while
\[
\min\big\{|\nabla{} u_\Lambda|-\mu\,  u_\Lambda,-\Delta_{\infty} u_\Lambda\big\}\geq0\quad\Rightarrow\quad \min\big\{|\nabla{} u_\Lambda|-\mu\,  w_1,-\Delta_{\infty} u_\Lambda\big\}\geq0.
\]
As $w_1= u_\Lambda=0$ on $\partial\Omega$, by comparison and the Perron method, we obtain that there exists a unique $w_2$ satisfying
\[
 v \leq w_1\leq w_2\leq u_\Lambda\quad\text{in}\ \Omega.
\]

Iterating this procedure, we construct a non-decreasing sequence 
\[
 v \leq w_1\leq w_2\leq\ldots \leq w_{k-1}\leq w_k\leq u_\Lambda
\]
of solutions of 
\begin{equation}\label{cap3:eq:exSI:eq:22}
\left\{
\begin{split}
&\min\big\{|\nabla{} w_k|-\mu\,  w_{k-1},-\Delta_{\infty} w_k\big\}=0&&\text{in}\ \Omega\\
&w_k=0&&\text{on}\ \partial\Omega.
\end{split}
\right.
\end{equation}
Notice that $\|w_k\|_{\infty}$ is uniformly bounded by construction. On the other hand, as  $-\Delta_\infty w_k\geq0$ in $\Omega$, we have  (see \cite{lindqvist-manfredi1,lindqvist-manfredi2} and also \cite{Juutinen} for a related construction) that
\[
|\nabla w_k(x)|\leq\frac{w_k(x)}{\textnormal{dist}(x,\partial\Omega)}\leq\frac{{u}_\Lambda(x)}{\textnormal{dist}(x,\partial\Omega)}\qquad \text{a.e.}\ x\in\Omega,
\]
for all $k>1$.
From there, both $\|w_k\|_{\infty}$ and $\|\nabla w_k\|_{\infty}$ are uniformly bounded in compact subsets of $\Omega$. We observe that $v,u_\Lambda$ are barriers in $\partial\Omega$ for each $w_k$. Hence by the Ascoli-Arzela theorem and the monotonicity of the sequence $\{w_k\}$, the whole sequence converges uniformly in
$\Omega$ to some $w\in{C}(\Omega)$ which verifies $w=0$ on $\partial\Omega$. Then, we can take limits in the viscosity sense in  \eqref{cap3:eq:exSI:eq:22} and obtain that the limit  $w$ is a viscosity solution of  \eqref{cafeslaestrella}, which proves the claim.
\end{proof}

\medskip


\section{Existence of a branch of minimal solutions for the limit problem}\label{existence.minimal.limit.problem}
In this section we  show that  for every $\Lambda\in(0,\Lambda_{\max}]$ there is a minimal solution 
of the  problem 
\begin{equation}\label{eq:limite.existence.minimal}
\left\{
\begin{aligned}
&\min\big\{|\nabla{}u|-\Lambda\,e^{u},
-\Delta_{\infty}u\big\}=0 &&\textrm{in}\ \Om,\\
&u=0 && \textrm{on}\ \partial\Omega.
\end{aligned}
\right.
\end{equation}
The  proof is based on the ideas in \cite{GP2}, although our construction is  different in order to take advantage of Corollary \ref{corolario.unicidad.limit.concavo}, our result of uniqueness for small solutions (the construction in \cite{GP2} would only allow us to conclude that the minimal solution satisfies  $\|u\|_\infty\leq \|u_{\Lambda_{\max}}\|_\infty=1$, and Corollary \ref{corolario.unicidad.limit.concavo} requires a strict inequality).

\begin{theorem}\label{theorem:exist.for.Lambda.max}
Let $\Omega\subset\R^n$ be a bounded domain. Then, problem \eqref{eq:limite.existence.minimal} has a minimal solution $u_\Lambda$ 
 for every $\Lambda\in(0,\Lambda_{\max}]$, where  $\Lambda_{\max}$ is given by \eqref{defn.lambda.hat}. Moreover, 
 \begin{enumerate}\itemsep3pt
 \item We have the estimate
 \[
 \Lambda \,\textnormal{dist}(x,\partial\Omega)
 \leq
 u_\Lambda(x)
 \leq
 e\Lambda \,\textnormal{dist}(x,\partial\Omega).
 \]
In particular, $\|u_\Lambda\|_\infty\leq e\Lambda\Lambda_1(\Omega)^{-1}<1$  for $\Lambda\in(0,\Lambda_{\max})$. 
 
 \item For every $\Lambda\in(0,\Lambda_{\max})$, $u_\Lambda$ is the only solution of  \eqref{eq:limite.existence.minimal} with 
  $\|u\|_\infty<1$.

\item The branch of minimal solutions is a non-decreasing continuum, in the sense that  if
$0<\Lambda<\Upsilon<\Lambda_{\max}$, then 
$u_{\Lambda}\leq u_{\Upsilon}$ and whenever $\Upsilon\to\Lambda\in(0,\Lambda_{\max})$, then $u_\Upsilon\to u_\Lambda$ uniformly.
\end{enumerate}

%
%
%
%

\end{theorem}


\begin{proof}
\noindent1.\quad{}Let $\underline{u}$ and $\overline{u}$ be the unique viscosity solutions of
\begin{equation}\label{eqn:low.bar.u}
\left\{
\begin{aligned}
&\min\Big\{|\nabla \underline{u}|-\Lambda,
-\Delta_{\infty}\underline{u}\Big\}=0&&\text{in}\ \Omega\\
&\underline{u}=0&&\textrm{on}\ \partial\Omega
\end{aligned}
\right.
\end{equation}
and
\begin{equation}\label{eqn:up.bar.u}
\left\{
\begin{aligned}
&\min\Big\{|\nabla \overline{u}|-e\Lambda,
-\Delta_{\infty}\overline{u}\Big\}=0&&\text{in}\ \Omega\\
&\overline{u}=0&&\textrm{on}\ \partial\Omega,
\end{aligned}
\right.
\end{equation}
respectively.  By Proposition \ref{yllegokawohlymandoaparar}, we  have the explicit expressions
\begin{equation}\label{explicit.exprss}
 \underline{u}(x)=\Lambda\,\textnormal{dist}(x,\partial\Omega)\qquad\textrm{and}\qquad
  \overline{u}(x)=e\Lambda\,\textnormal{dist}(x,\partial\Omega)
\end{equation}
and  
$\underline{u}\leq\overline{u}$ follows trivially (alternatively, this  can be proved  by comparison, Lemma \ref{comparison.f(x)}, using that $\overline{u}$ is a viscosity supersolution of \eqref{eqn:low.bar.u}).


\medskip

\noindent2.\quad{}Define now $u_1$, viscosity solution of 
\begin{equation}\label{eqn:u1}
\left\{
\begin{aligned}
&\min\Big\{|\nabla u_1|-\Lambda \,e^{\underline{u}},
-\Delta_{\infty}u_{1}\Big\}=0&&\text{in}\ \Omega\\
&u_1=0&&\textrm{on}\ \partial\Omega.
\end{aligned}
\right.
\end{equation}
Let us show that
\begin{equation}\label{u1.bounds}
\underline{u}\leq u_1\leq\overline{u}\quad\text{in}\ \Omega. 
\end{equation}
First, we prove $u_1\leq\overline{u}$. We aim to show that 
$\min\big\{|\nabla u_1|-e\Lambda,
-\Delta_{\infty}u_1\big\}\leq0$ in the viscosity sense
and then apply comparison for equation \eqref{eqn:up.bar.u}, see 
Lemma \ref{comparison.f(x)}.  Therefore, let $x_0\in\Om$ and
$\phi\in{C}^2(\Om)$ such that $u_1-\phi$ attains a
 local maximum at $x_0$. 
We can assume that $-\Delta_{\infty}\phi(x_0)>0$ because we are done otherwise. Then, from
\eqref{eqn:u1}, \eqref{explicit.exprss}, and \eqref{defn.lambda.hat}, we have
\[
|\nabla \phi(x_0)|
\leq 
\Lambda \,e^{\underline{u}(x_0)}
\leq
\Lambda \,e^{ \Lambda_{\max}\,\Lambda_1(\Omega)^{-1}}
<
e\Lambda.  
\]
In order to show  that $u_1\geq\underline{u}$, we  prove  that 
$\min\big\{|\nabla u_1|-\Lambda,
-\Delta_{\infty}u_1\big\}\geq0$  in the viscosity sense
and then proceed by comparison for equation \eqref{eqn:low.bar.u}. Indeed, since $u_1$ is a supersolution of \eqref{eqn:u1}, we have $-\Delta_{\infty}u_{1}\geq0$ and
$|\nabla u_1|\geq \Lambda \,e^{\underline{u}}\geq \Lambda$
in the viscosity sense, as desired.

\medskip

\noindent3.\quad{}For each $k\geq 0$, we define $u_{k+1}$ as the viscosity solution of 
\begin{equation}\label{eqn:u.k+1}
\left\{
\begin{aligned}
&\min\Big\{|\nabla u_{k+1}|-\Lambda \,e^{u_{k}},
-\Delta_{\infty}u_{k+1}\Big\}=0&&\text{in}\ \Omega\\
&u_{k+1}=0&&\textrm{on}\ \partial\Omega
\end{aligned}
\right.
\end{equation}
with $u_0=\underline{u}$ and $u_1$ given by \eqref{eqn:u1}. Let us show that for all $k\geq0$
\begin{equation}\label{uk.increas.bdd}
\underline{u}\leq u_{k}\leq u_{k+1}\leq\overline{u}\qquad\text{in}\ \Omega,
\end{equation}
that is, the sequence $\{u_k\}_{k\geq0}$ is non-decreasing and uniformly bounded.

We prove \eqref{uk.increas.bdd} by induction. First, notice that \eqref{u1.bounds} proves the case when $k=0$. Assume \eqref{uk.increas.bdd} holds true for $k-1$ and let us  prove that $u_{k}\leq u_{k+1}$. Since  $u_{k+1}$ is, by definition, a viscosity supersolution of \eqref{eqn:u.k+1}, we have $-\Delta_{\infty}u_{k+1}\geq0$ and
$|\nabla u_{k+1}|\geq \Lambda \,e^{u_{k}}\geq \Lambda\,e^{u_{k-1}}$
in the viscosity sense
by the induction hypothesis.
Therefore, $u_{k+1}$ is a viscosity solution of 
\[
\min\big\{|\nabla u_{k+1}|-\Lambda\, e^{u_{k-1}},
-\Delta_{\infty}u_{k+1}\big\}\geq0\qquad \textrm{in}\ \Omega.
\]
By definition, we have $\min\big\{|\nabla u_{k}|-\Lambda\, e^{u_{k-1}},
-\Delta_{\infty}u_{k}\big\}=0
$
and $u_{k}\leq u_{k+1}$ follows by comparison, see 
Lemma \ref{comparison.f(x)} (notice that $e^{u_{k-1}}$ is bounded, positive, and continuous, since the  $\infty$-superharmonicity of $u_{k-1}$
imply its Lipschitz continuity, see \cite{lindqvist-manfredi2}).

To  prove  that $u_{k+1}\leq\overline{u}$,  we show that 
\[
\min\big\{|\nabla u_{k+1}|-e\Lambda,
-\Delta_{\infty}u_{k+1}\big\}\leq0\qquad\textrm{in} \ \Omega
\]
and use comparison for equation \eqref{eqn:up.bar.u} (see 
Lemma \ref{comparison.f(x)}).  Therefore, let $x_0\in\Om$ and
$\phi\in{C}^2(\Om)$ such that $u_{k+1}-\phi$ attains a
 local maximum at $x_0$.
Assume that $-\Delta_{\infty}\phi(x_0)>0$ since we are done otherwise. Then, from \eqref{eqn:u.k+1}, \eqref{explicit.exprss},  \eqref{defn.lambda.hat},  and the induction hypothesis we get
\[
|\nabla \phi(x_0)|
\leq 
\Lambda \,e^{u_{k}(x_0)}
\leq
\Lambda\,e^{\|\overline{u}\|_\infty}
\leq
e\Lambda.
\]

 \medskip
 
 \noindent4.\quad{}We have obtained a non-decreasing sequence $\{u_k\}_{k\geq0}$, uniformly bounded by $\underline{u}$ and $\overline{u}$ given by \eqref{explicit.exprss}. Therefore, we can pass to the limit in the viscosity sense in the same way as in Proposition \ref{prop:NO.exist.beyond.Lambda.hat} and get a viscosity solution $u_\Lambda$ of problem \eqref{eq:limite.existence.minimal} as intended.  It is also clear that the solution $u_\Lambda$ we just found is minimal for every $\Lambda\in(0,\Lambda_{\max}]$, because any solution of  \eqref{eq:limite.existence.minimal} could be taken as $\overline{u}$ in the iteration (note that the function $u_\Lambda$ does not depend on $\overline{u}$). Moreover, by \eqref{explicit.exprss} and \eqref{uk.increas.bdd}, we have
 \[
 \Lambda\,\textnormal{dist}(x,\partial\Omega)
\leq u_{\Lambda}(x)\leq
e\Lambda\,\textnormal{dist}(x,\partial\Omega)
\qquad\text{in}\ \Omega.
 \]
Therefore, by  Corollary \ref{corolario.unicidad.limit.concavo}, $u_\Lambda$ is the only solution of \eqref{eq:limite.existence.minimal} with 
  $\|u\|_\infty< 1$ for every $\Lambda\in(0,\Lambda_{\max})$.

  \medskip
 
 \noindent5.\quad{}  Let us prove that the branch of minimal solutions is non-decreasing, i.e., 
$u_{\Lambda}\leq u_{\Upsilon}$ whenever $0<\Lambda<\Upsilon<\Lambda_{\max}$. To this aim, let us just observe that we can repeat the above construction  taking $\overline{u}=u_\Upsilon$ and keeping $ \underline{u}(x)=\Lambda\,\textnormal{dist}(x,\partial\Omega)$ as  before. In this way, we recover the minimal solution $u_\Lambda$ with the estimate $u_{\Lambda}\leq u_{\Upsilon}<1$.

 We conclude by showing that the branch of minimal solutions is a continuum.  Arguing again as in the proof of Proposition \ref{prop:NO.exist.beyond.Lambda.hat}, we see that, for every $\Lambda\in(0,\Lambda_{\max})$, the uniform limits
\[
\widehat{u}_\Lambda=\lim_{\Upsilon\to\Lambda^+}u_\Upsilon,\qquad \textnormal{and}\qquad\widecheck{u}_\Lambda=\lim_{\Upsilon\to\Lambda^-}u_\Upsilon
\]
are both viscosity solutions of \eqref{eq:limite.existence.minimal} with   $\max\big\{\|\widehat{u}_\Lambda\|_\infty,\|\widecheck{u}_\Lambda\|_\infty\big\}< 1$. Therefore $\widehat{u}_\Lambda\equiv\widecheck{u}_\Lambda$ by    Corollary ~\ref{corolario.unicidad.limit.concavo}, as desired.
\end{proof}


\section{Minimal solutions  achieved as limits of $p$-minimal solutions as $p\to\infty$}\label{p.minimal.limits}

This section shows that uniform limits  of appropriately scaled, minimal solutions  of
\begin{equation}\label{main.problem.finite.p.minimal}
\left\{
\begin{aligned}
-&\Delta_{p} u=\lambda\,e^{u}&&\textrm{in}\ \Omega\subset\R^n\\
&u=0 &&\textnormal{on}\ \partial\Omega
\end{aligned}
\right.
\end{equation}
converge  to the minimal solutions of the limit problem \eqref{eq:limite.existence.minimal}, found in Section ~\ref{existence.minimal.limit.problem}.
Observe that the fact that the limit solution is  minimal is nontrivial; in principle, a limit solution could be different from the minimal one.
Here is where we use the uniqueness results  from Section \ref{comparison}. We prove the following.

\begin{theorem}\label{convergencia.minimales}
Let $\Lambda\in(0,\Lambda_{\max})$, and $\{\lambda_{p}\}_{p}$ be a sequence such that 
\[
\lim_{p\to\infty}\frac{\lambda_{p}^{1/p}}{p}=\Lambda.
\] 
For each $\lambda_p$, 
consider   $u_{\lambda_p,p}$, the minimal  solution of \eqref{main.problem.finite.p.minimal} for $\lambda=\lambda_p$.
Then, 
\[
\frac{u_{\lambda_{p},p}}{p}\to u_\Lambda,\quad\text{uniformly as}\ p\to\infty,
\]
where $u_\Lambda$ is the minimal solution of the limit problem \eqref{eq:limite.existence.minimal}.
\end{theorem}

We devote the rest of the section to the proof of Theorem \ref{convergencia.minimales}. In order to obtain estimates that allow us to pass to the limit, we provide an explicit construction of the branch of minimal  solutions of \eqref{main.problem.finite.p.minimal}.
Although these  are rather classic facts, see \cite{GP2, GPP}, some of our results appear to be new. Additionally,  we provide a modified, more streamlined, and systematic construction that exhibits  the  dependences on $p$ at each step, which is necessary in order to pass to the limit.

First, we show that problem \eqref{main.problem.finite.p.minimal} has a minimal solution up to a certain, explicit $\widecheck{\lambda}_{p}$.

\begin{prop}\label{prop:exist.for.lambda.hat.finite.p}
Let $\Omega\subset\R^n$ be a bounded domain and $p>n$. Then,
problem \eqref{main.problem.finite.p.minimal}
 has a minimal solution $u_{\lambda,p}(x)$ for every $\lambda\in(0,\widecheck{\lambda}_{p}]$, where
\begin{equation}\label{lambda0p}
\widecheck{\lambda}_{p}
=\left(\frac{p-1}{e\|v_{p}\|_\infty}\right)^{p-1}
\end{equation}
and $v_{p}$ is given by \eqref{aux.kawohl2}.
Moreover, 
\begin{enumerate}\itemsep3pt
 \item  For every $\lambda\leq\widecheck{\lambda}_{p}$, we have the estimate
\begin{equation} \label{bounds.for.m.lambda.p}
 \lambda^\frac{1}{p-1}  v_{p}(x)
 \leq
u_{\lambda,p}(x)
 \leq
e\, \lambda^\frac{1}{p-1}  v_{p}(x)\quad\textrm{in}\ \overline\Omega.
\end{equation}
 
 \item  For every $\lambda\leq\widecheck{\lambda}_{p}$, the minimal solution $u_{\lambda,p}$ is the only solution of  \eqref{main.problem.finite.p.minimal} with   $\|u\|_\infty\leq p-1$.

\item The branch of minimal solutions is non-decreasing, in the sense that  if
$0<\lambda<\mu\leq\widecheck
{\lambda}_{p}$, then 
$u_{\lambda,p}\leq u_{\mu,p}$ in $\Omega$.
\end{enumerate}

\end{prop}

The uniqueness result in part 2  of  Proposition \ref{prop:exist.for.lambda.hat.finite.p} appears to be new. For the proof, we use the following  comparison principle, an adaptation of \cite[Lemma 4.1]{Abdellaoui.Peral.2003} to problems that are proper (in the sense of \cite{CIL}) only for ``small" sub- and supersolutions.

\begin{lemma}\label{theorem.ppio.comparacion.plap}
 Let $p>1$ and $f:\mathbb{R}\to\mathbb{R}$ be a non-negative continuous function  for which there exists $c\in(0,\infty]$  such that 
  \begin{equation*}
\frac{f(t)}{t^{p-1}}\quad\text{is non-increasing for all}\ t\in(0,c).
\end{equation*}
Assume that $u,v\in W_0^{1,p}(\Omega)\cap C^1(\Omega)$ are positive in $\Omega$, $\max\{\|u\|_\infty,\|v\|_\infty\}\leq c$ and 
\[
-\Delta_pu\leq f(u)\quad \textrm{and}\quad -\Delta_pv\geq f(v) \qquad\textrm{in}\ \Omega.
\]
Then $u\leq v$ in $\Omega$.
\end{lemma}

We omit the proof of the lemma since it is a straightforward modification of
\cite[Lemma 4.1]{Abdellaoui.Peral.2003} 
 (note that  $c=\infty$ in \cite{Abdellaoui.Peral.2003}). We proceed now with the proof of Proposition \ref{prop:exist.for.lambda.hat.finite.p}.

\begin{proof}[Proof of Proposition \ref{prop:exist.for.lambda.hat.finite.p}]
\noindent1.\quad{}Consider $\underline{u}$ and $\overline{u}$, the respective solutions of
\[
\left\{
\begin{aligned}
-&\Delta_p \underline{u}=\lambda&&\text{in}\ \Omega\\
&\underline{u}=0&&\textrm{on}\ \partial\Omega
\end{aligned}
\right.
\]
and
\[
\left\{
\begin{aligned}
-&\Delta_p  \overline{u}=\lambda \,e^{p-1}&&\text{in}\ \Omega\\
&\overline{u}=0&&\textrm{on}\ \partial\Omega.
\end{aligned}
\right.
\]
By the weak comparison principle for the $p$-Laplacian, we have that
$0\leq \underline{u}\leq\overline{u}$ in $\overline\Omega.$
Define now $u_1$,  solution of 
\begin{equation}\label{eqn:u1.p}
\left\{
\begin{aligned}
-&\Delta_p u_1=\lambda \,e^{\underline{u}}&&\text{in}\ \Omega\\
&u_1=0&&\textrm{on}\ \partial\Omega.
\end{aligned}
\right.
\end{equation}
We clearly have $-\Delta_p u_1\geq\lambda=-\Delta_p \underline{u}$. On the other hand, we find $\underline{u}=\lambda^\frac{1}{p-1}  v_{p}$ by rescaling, which together with \eqref{lambda0p} yields
\[
-\Delta_p u_1
\leq
\lambda \,e^{\|\underline{u}\|_\infty}
\leq
\lambda \,e^{(\widecheck\lambda_p)^{1/(p-1)}  \|v_{p}\|_\infty}
=
\lambda \,e^{(p-1)/e}
\leq
-\Delta_p  \overline{u}.
\]
Then, by the weak comparison principle we have
$\underline{u}\leq u_1\leq\overline{u}$ in $ \overline\Omega.$

\medskip

\noindent2.\quad{}Now, for each $k\geq 1$ define $u_{k+1}$,  solution of 
\[
\left\{
\begin{aligned}
-&\Delta_p u_{k+1}=\lambda \,e^{u_{k}}&&\text{in}\ \Omega\\
&u_{k+1}=0&&\textrm{on}\ \partial\Omega
\end{aligned}
\right.
\]
with  $u_1$ defined by \eqref{eqn:u1.p}. 
Let us show by induction that 
\[
\underline{u}\leq u_{k}\leq u_{k+1}\leq\overline{u}\quad\text{in}\ \overline\Omega
\]
for all $k\geq1$.
 It is easy to see that 
$\underline{u}\leq u_{k}\leq u_{k+1}$ by comparison. To prove $u_{k+1}\leq\overline{u}$,  notice that the induction hypothesis, the rescaling
$\overline{u}=\lambda^\frac{1}{p-1} e\, v_{p}$, and  \eqref{lambda0p}  yield
\[
-\Delta_p u_{k+1}
=
\lambda \,e^{u_{k}}
\leq
\lambda \,e^{\|\overline{u}\|_\infty}
\leq
\lambda \,e^{(\widecheck\lambda_p)^{1/(p-1)} e\, \|v_{p}\|_\infty}
=
 \lambda\, e^{p-1}
=-\Delta_p  \overline{u}.
\]
Then,  $u_{k+1}\leq\overline{u}$ follows by comparison.

\medskip

\noindent3.\quad{}We have obtained an increasing sequence $\{u_k\}_{k\geq0}$, uniformly bounded by $\underline{u}$ and $\overline{u}$.
 Therefore, we can pass to the limit and get a solution $u_{\lambda,p}$ that satisfies the bounds \eqref{bounds.for.m.lambda.p}. It is also clear that $u_{\lambda,p}$ is minimal, because any solution of \eqref{main.problem.finite.p.minimal} could be taken as $\overline{u}$ in the iterative scheme (note that each $u_k$ does not depend on $\overline{u}$). 
Similarly, we see that the branch of minimal solutions is non-decreasing, since whenever $\lambda<\mu$, we can take $\overline{u}=u_{\mu,p}$ in the construction of $u_{\lambda,p}$ and obtain $u_{\lambda,p}\leq u_{\mu,p}$. 

\medskip

\noindent4.\quad{}Finally, let us  denote $f(t)=\lambda\,e^t$. It is elementary to see that $f(t)/t^{p-1}$ is non-increasing for $0<t< p-1$. Moreover, by \eqref{lambda0p} and \eqref{bounds.for.m.lambda.p} we have that $\|u_{\lambda,p}\|_\infty\leq p-1$.
Therefore, we can apply 
 Lemma \ref{theorem.ppio.comparacion.plap} with $c=p-1$ and conclude  that  $u_{\lambda,p}$ is the only solution of  \eqref{main.problem.finite.p.minimal} with   $\|u\|_\infty\leq p-1$ for every $\lambda\in(0,\widecheck{\lambda}_{p}]$.
\end{proof}

The next result states  that problem \eqref{main.problem.finite.p.minimal}
has no solution for large $\lambda$; that is, there is a value  $\widehat{\lambda}_p>0$ such that \eqref{main.problem.finite.p.minimal} has no weak solution  with $\lambda>\widehat{\lambda}_p$.

\begin{prop}[{\cite[Theorem 2.1]{GP2}}]\label{existNOpp}
Problem \eqref{main.problem.finite.p.minimal}
does not have a solution for $\lambda
>
\widehat{\lambda}_p$, where
\begin{equation}\label{lambda.hat.p}
\widehat{\lambda}_p
=
\lambda_1(p,\Omega)
\cdot
\max
\left\{
1,
\left(\frac{p-1}{e}\right)^{p-1}
\right\}.
\end{equation}
\end{prop}



At this point we can define
\begin{equation}\label{defn.lambda.max.p}
\lambda_{\max,p}
=
\sup\big\{
\lambda>0\ : \ \textnormal{problem}\ \eqref{main.problem.finite.p.minimal}\ \textnormal{has a solution}
\big\}.
\end{equation}
In the next result we show that $\lambda_{\max,p}$  is well-defined, find its asymptotic behavior as $p\to\infty$, and complete the construction of the branch of minimal solutions.

\begin{prop}
Let  $\Omega\subset\R^n$ be a bounded domain and $p>n$. 
Then, $\lambda_{\max,p}$ 
given by
\eqref{defn.lambda.max.p} is well-defined (in the sense that it is positive and finite). Moreover, 
  \eqref{main.problem.finite.p.minimal} has a minimal solution $u_{\lambda,p}(x)$ for every $\lambda\in(0,\lambda_{\max,p})$ and no solution for $\lambda>\lambda_{\max,p}$. In addition,
\begin{equation}\label{estimate.lambda.max.p}
\widecheck{\lambda}_{p}
\leq\lambda_{\max,p}\leq\widehat{\lambda}_{p},
\end{equation}
where $\widecheck
\lambda_{p}$ and $\widehat{\lambda}_{p}
$
are respectively given by  \eqref{lambda0p}, \eqref{lambda.hat.p}, and 
\[
 \lim_{p\to\infty}\frac{\lambda_{\max,p}^{1/p}}{p}
 =
\Lambda_{\max}
\]
for $\Lambda_{\max}$ defined by \eqref{defn.lambda.hat}.
\end{prop}

\begin{proof}
By Propositions \ref{prop:exist.for.lambda.hat.finite.p} and \ref{existNOpp}, we have that $0<\widecheck\lambda_p\leq\lambda_{\max,p}\leq\widehat{\lambda}_{p}<\infty$. 
 Moreover, although we do not know $\lambda_{\max,p}$  explicitly,  \eqref{lambda0p}, \eqref{lambda.hat.p}, and  \eqref{estimate.lambda.max.p}, along with Proposition \ref{yllegokawohlymandoaparar} and  Lemma \ref{lema.autovalores}
 provide its asymptotic behavior, namely,
\[
 \lim_{p\to\infty}\frac{\lambda_{\max,p}^{1/p}}{p}=\lim_{p\to\infty}\frac{\widecheck
{\lambda}_{p}
^{1/p}}{p}=\lim_{p\to\infty}\frac{\widehat{\lambda}_{p}^{1/p}}{p}
=e^{-1}\Lambda_{1}(\Omega)=
\Lambda_{\max}.
\]

Let us now complete the construction of the branch of minimal solutions. Since $\lambda_{\max,p}<\infty$
we can take  $\mu$ arbitrarily close to $\lambda_{\text{max},p}$  and  $u_{\mu}$ solution of
\begin{equation*}
\left\{\begin{split}
-&\Delta_{p}u_{\mu}=\mu\,e^{u_{\mu}}&&\text{in}\ \Omega,\\
&u_{\mu}=0&&\text{on}\ \partial\Omega.
\end{split}
\right.
\end{equation*}
Then, for every $\lambda\in(\widecheck\lambda_p,\mu]$ we can produce a minimal solution as in Proposition  \ref{prop:exist.for.lambda.hat.finite.p}, taking $\overline{u}=u_\mu$
 in the iteration.
\end{proof}



We are now ready to prove Theorem \ref{convergencia.minimales}.

\begin{proof}[Proof of Theorem \ref{convergencia.minimales}]
\noindent1.\quad{}We have that 
\[
\left\{
\begin{split}
-&\Delta_{p}u_{\lambda_p,p}=\lambda_p\, e^{u_{\lambda_p,p}}&&\text{in}\ \Omega\\
&u_{\lambda_p,p}=0 &&\text{on}\ \partial\Omega.
\end{split}
\right.
\]
Multiplying  the equation by
$u_{\lambda_p,p}$ and integrating by parts, we get
\[
\int_{\Omega}|\nabla{}u_{\lambda_p,p}(x)|^p\,dx
=
\lambda_p
\int_{\Omega}u_{\lambda_p,p}(x)\,e^{u_{\lambda_p,p}(x)}\,dx.
\]
Let us fix $p>n$. Then, for every  $m\in(n,p)$ and every $x,y\in\Omega$, there exists a positive constant $C$ independent of $m,p$ (see \cite[Lemma 3.3]{Charro.Parini}) such that 
\begin{multline}\label{morrey.main.problem}
\frac{|u_{\lambda_p,p}(x)-u_{\lambda_p,p}(y)|}{|x-y|^{1-\frac{n}{m}}}
\leq
{C}\left(\int_{\Omega}\left|\nabla{}u_{\lambda_p,p}\right|^mdx\right)^{1/m}\leq
{C}\,|\Omega|^{\frac1m-\frac1p}
\left(\int_{\Omega}\left|\nabla{}u_{\lambda_p,p}\right|^pdx\right)^{1/p}
\\
=
\displaystyle{{C}\,|\Omega|^{\frac1m-\frac1p}
\left(
\lambda_p
\int_{\Omega}u_{\lambda_p,p}\,e^{u_{\lambda_p,p}}\,dx
\right)^{1/p}}
\leq
{C}\,
\Big(
\lambda_p\|u_{\lambda_p,p}\|_\infty \,e^{\|u_{\lambda_p,p}\|_\infty}
\Big)^{1/p}.
\end{multline}
Let us now find estimates for $\|u_{\lambda_p,p}\|_\infty$.

\medskip

\noindent2.\quad{}Consider $\widecheck{\lambda}_{p}$, given by 
\eqref{lambda0p}. Since
\[
\lim_{p\to\infty}
\frac{\widecheck{\lambda}_{p}^{1/p}}{p}
=
\Lambda_{\max}>\Lambda
=
\lim_{p\to\infty}\frac{\lambda_{p}^{1/p}}{p},
\]
   there exists $p_0$ such that $\lambda_p<\widecheck{\lambda}_{p}$ for all  $p\geq p_0$.
Then, by estimate ~\eqref{bounds.for.m.lambda.p}, we have
\[
 \frac{\lambda_p^\frac{1}{p-1}}{p}  v_{p}(x)
 \leq
\frac{u_{\lambda_p,p}(x)}{p}
 \leq
\frac{\lambda_p^\frac{1}{p-1}}{p}\, e\, v_{p}(x)\quad\textrm{in}\ \overline\Omega,
\]
where $v_{p}$ is given by \eqref{aux.kawohl2}. 
Take $\epsilon>0$ such that $(1+\epsilon)\Lambda<\Lambda_{\max}$.
By Proposition \ref{yllegokawohlymandoaparar},  we know that $v_{p}\to\text{dist}(\cdot,\partial\Omega)$ uniformly as $p\to\infty$  and we deduce that 
\begin{equation}\label{estimate.in.the.limit}
\frac{\|u_{\lambda_{p},p}\|_\infty}{p}\leq  (1+\epsilon)\Lambda e\|\text{dist}(\cdot,\partial\Omega)\|_\infty
=(1+\epsilon)\Lambda\Lambda_{\max}^{-1}<1
\end{equation}
for  $p$ large enough.
Then, from \eqref{morrey.main.problem} and  the  Arzel\`a-Ascoli theorem, we find that there exists a subsequence  $p'$ and a limit function   $u_\Lambda$ such that 
\[
\frac{u_{\lambda_{p'},p'}}{p'}\to u_\Lambda,\quad\text{uniformly as}\ p'\to\infty.
\]

\medskip

\noindent3.\quad{}By Proposition \ref{proposition.limit.eq}, we have that $u_\Lambda$ is a  viscosity solution of the limit problem
\eqref{eq:limite.existence.minimal}.
Additionally,  from estimate \eqref{estimate.in.the.limit} we deduce
$\|u_\Lambda\|_\infty\leq \Lambda\Lambda_{\max}^{-1}<1$, and then
 Theorem ~\ref{theorem:exist.for.Lambda.max} implies that $u_\Lambda$ must be the minimal solution of the limit problem \eqref{eq:limite.existence.minimal}.
Therefore, the whole sequence $u_{\lambda_p,p}$ converges, and not only a subsequence, which concludes the proof.
 \end{proof}

\medskip


\section{Multiplicity results in special domains}\label{section.explicit}

This section proves that, under certain geometric assumptions on the domain $\Omega$, it is possible to compute an explicit curve of solutions. Moreover, we establish a further non-existence result with the aid of this curve of solutions. 
  To this aim, we consider the \emph{ridge set} of $\Omega$,
\begin{eqnarray*}
\mathcal{R}\!\!\!&=&\!\!\!\{x\in\Omega:
\textnormal{dist}(x,\partial\Omega)\
\text{is not differentiable at } x\}\\
\!\!\!&=&\!\!\!\{x\in\Omega:\ \exists\,x_1,x_2\in\partial\Omega,\
x_1\neq{}x_2,\ \text{s.t.}\
|x-x_1|=|x-x_2|=\textnormal{dist}(x,\partial\Omega)\}
\end{eqnarray*}
and its subset $\mathcal{M}$, the \emph{set of maximal distance} to the boundary,
\[
\mathcal{M}=\big\{x\in\Omega:\textnormal{dist}(x,\partial\Omega)=\|\textnormal{dist}(\cdot,\partial\Omega)\|_\infty\big\}.
\]

We have proved in Theorem \ref{theorem:exist.for.Lambda.max} the existence of minimal solutions for the limit problem  \eqref{limite1}, as well as several non-existence results in Propositions  \ref{noexist.small} and \ref{prop:NO.exist.beyond.Lambda.hat}.
These results hold for general bounded domains $\Omega$. 
In this section, we find a second solution to the limit problem \eqref{limite1} under the additional assumption $\mathcal{M}\equiv\mathcal{R}$. Furthermore, both solutions lie on an explicit curve of solutions. Some examples of domains satisfying $\mathcal{M}\equiv\mathcal{R}$ are the ball, the annulus, and the stadium (convex hull of two balls of the same radius). A square or an ellipse does not verify the condition.





\begin{figure}[h!]
\begin{center}
\resizebox{0.4\textwidth}{!}{%
\includegraphics{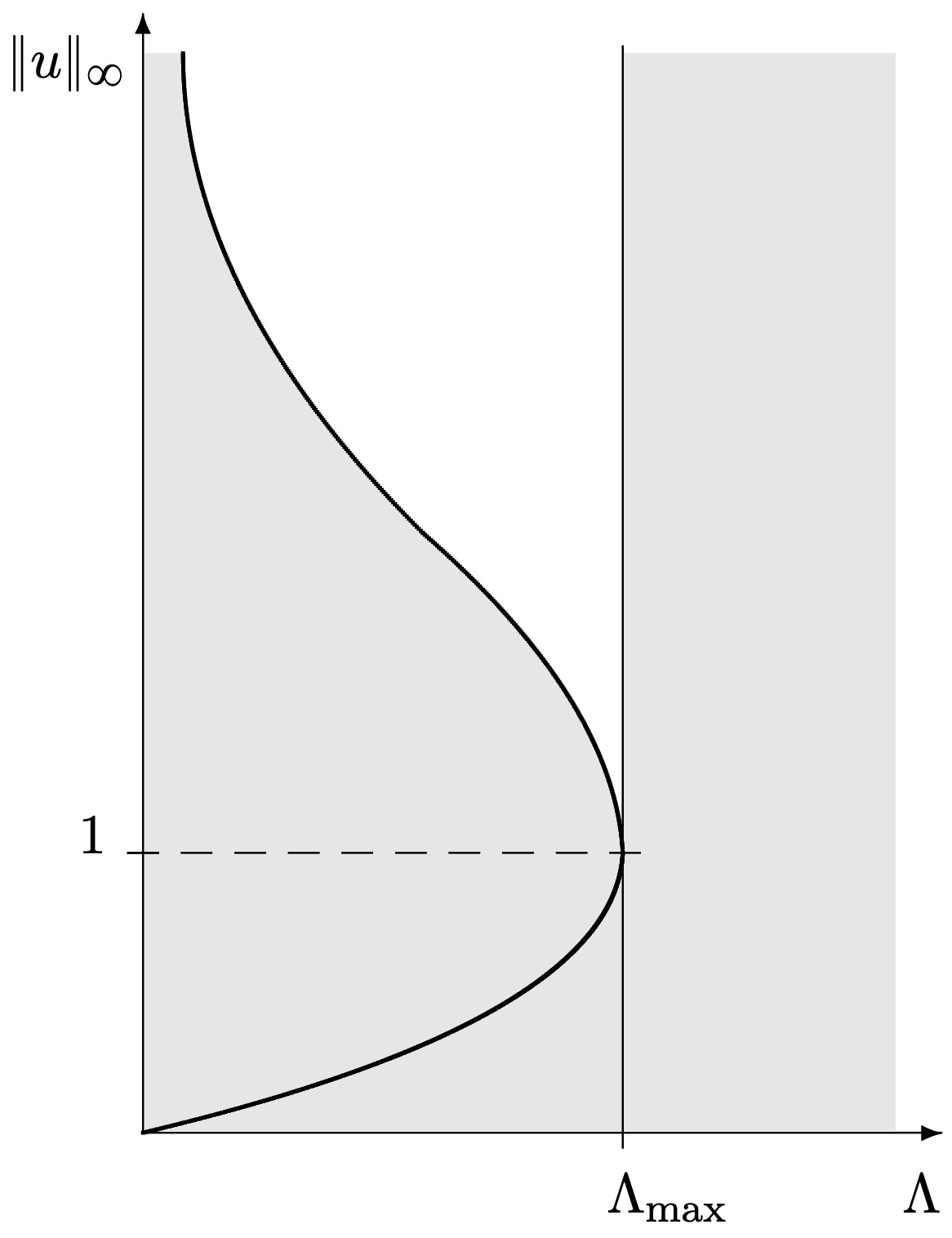} }
\caption{Curve of explicit solutions $\Lambda_1(\Omega)\,\|u_\Lambda\|_\infty-\Lambda\,e^{\|u_\Lambda\|_\infty}=0$ in Theorem \ref{conecone} and
regions of non-existence derived from Proposition \ref{prop:NO.exist.beyond.Lambda.hat}, Theorem \ref{asereje},  and the uniqueness result in Theorem \ref{theorem:exist.for.Lambda.max}.
}
\end{center}
\end{figure}


\subsection{A curve of explicit solutions}
 We have the following result.

\begin{theorem}\label{conecone}
Let $\Lambda>0$ and $\Lambda_{\max}$ given by \eqref{defn.lambda.hat}. Assume that $\Omega\subset\mathbb{R}^n$ is a bounded domain that satisfies
$\mathcal{M}\equiv\mathcal{R}$. Let us consider solutions of the form 
\begin{equation}\label{cono.bis2}
u(x)=\alpha\cdot\textnormal{dist}(x,\partial\Omega),\qquad \alpha>0
\end{equation}
for the problem
\begin{equation}\label{problema.exponencial.infty}
\left\{
\begin{aligned}
&\min\big\{|\nabla u(x)|-\Lambda\,e^{u(x)},
-\Delta_{\infty}u(x)\big\}=0&&\textrm{in}\ \Om,\\
&u=0&& \text{on}\ \partial\Omega.
\end{aligned}
\right.
\end{equation}
Then, problem \eqref{problema.exponencial.infty}
\begin{itemize}
 \item[$i)$]Has two  solutions of the form \eqref{cono.bis2} if $0<\Lambda<\Lambda_{\max}$, corresponding to the two roots of
\begin{equation}\label{chunga.bis2}
\alpha-\Lambda\,e^{\alpha\,\|\textnormal{dist}(\cdot,\partial\Omega)\|_\infty} =0.
\end{equation}
 \item[$ii)$] Has one  solution of the form \eqref{cono.bis2} for $\Lambda=\Lambda_{\max}$, with $\alpha=\|\textnormal{dist}(\cdot,\partial\Omega)\|_\infty^{-1}$.
 \item[$iii)$] Has no  solutions for $\Lambda>\Lambda_{\max}$, and only the trivial solution for $\Lambda=0$.
\end{itemize}
\end{theorem}

\begin{remark}
By Theorem \ref{theorem:exist.for.Lambda.max}, 
for $0<\Lambda<\Lambda_{\max}$ the solution of the form \eqref{cono.bis2} with  smallest $\alpha$ is the minimal solution of \eqref{problema.exponencial.infty}.
\end{remark}

\begin{proof}
First of all, we are going to check that
\[
-\Delta_{\infty}u(x)=0\qquad\text{in}\ \Omega\setminus\mathcal{R}
\]
in the viscosity sense. Let $\phi\in{C}^2$ and $x_0\in\Omega\setminus\mathcal{R}$ such that $u-\phi$ has a local maximum at $x_0$. We can assume $u(x_0)=\phi(x_0)$ and $\nabla\phi(x_0)\neq0$.
A Taylor expansion, and the fact that $\phi$ touches $u$ from above at $x_0$ yield
\[
-\frac{\Delta_\infty\phi(x_0)}{|\nabla \phi(x_0)|^2}+o(1)\leq\frac{1}{\epsilon^2}\left(2u(x_0)-\max_{y\in B_\epsilon(x_0)}u(y)-\min_{y\in B_\epsilon(x_0)}u(y)\right)
\]
as $\epsilon\to0$.  From \eqref{cono.bis2} we have that
\[
\max_{y\in B_\epsilon(x_0)}u(y)=u(x_0)+\alpha\epsilon,\qquad\min_{y\in B_\epsilon(x_0)}u(y)=u(x_0)-\alpha\epsilon
\]
and we deduce that $u$ is $\infty$-subharmonic in $\Omega\setminus\mathcal{R}$. The proof that it is also $\infty$-superharmonic is analogous.
 Hence, we need make sure that
\[
|\nabla u(x)|-\Lambda\,e^{u(x)}\geq 0\qquad\forall x\in\Omega\setminus\mathcal{R}.
\]
Indeed, we find that
\[
|\nabla u(x)|-\Lambda\,e^{u(x)}=\alpha-\Lambda\,e^{\alpha\cdot\textnormal{dist}(x,\partial\Omega)}
\]
(recall that $x\notin\mathcal{R}$ and the derivatives are classical).
Since we can choose points $x\notin\mathcal{R}\equiv\mathcal{M}$ arbitrarily close to $\mathcal{M}$, we find the necessary condition 
\begin{equation}\label{cucu.bis2}
\alpha-\Lambda\,e^{\alpha\cdot\|\textnormal{dist}(\cdot,\partial\Omega)\|_\infty}\geq0.
\end{equation}

\medskip

Next, we turn our attention to the ridge set $\mathcal{R}$. First, observe that cones as in  
\eqref{cono.bis2} are always supersolutions of \eqref{problema.exponencial.infty} in the ridge set, since they  cannot be touched from below with  ${C}^2$ functions  at those points. Hence, we only have to consider the subsolution case. So, let $x_0\in\mathcal{R}$ and $\phi\in{C}^2$ such that $u-\phi$ has a local maximum point at $x_0$. We aim to prove that
\begin{equation}\label{elgol.bis2}
\min\big\{|\nabla \phi(x_0)|-\Lambda\,e^{u(x_0)},-\Delta_{\infty}\phi(x_0)\big\}\leq0.
\end{equation}

It is well-known (see for instance \cite[Lemma 6.10]{JuutinenTESIS}) that
\[
\min\Big\{|\nabla{}u(x)|-\alpha,
-\Delta_{\infty}u(x)\Big\}=0
\]
in the viscosity sense.
Thus, by definition of viscosity subsolution we have that either $|\nabla{}\phi(x_0)|\leq \alpha$ or $-\Delta_{\infty}\phi(x_0)\leq0$. In the latter case, \eqref{elgol.bis2} holds and there is nothing to prove. Thus, we can assume in the sequel that $-\Delta_\infty\phi(x_0)>0$ and $|\nabla{}\phi(x_0)|\leq \alpha$. Then, since $x_0\in\mathcal{R}\equiv\mathcal{M}$, we have $u(x_0)=\alpha\,\|\textnormal{dist}(\cdot,\partial\Omega)\|_\infty$ and
\[
|\nabla \phi(x_0)|-\Lambda\,e^{u(x_0)}\leq \alpha-\Lambda\,e^{\alpha\,\|\textnormal{dist}(\cdot,\partial\Omega)\|_\infty}.
\]
Recalling \eqref{cucu.bis2}, we discover that the only possibility is that 
\eqref{chunga.bis2} holds.
The rest of the proof is devoted to study the number of positive solutions of  equation \eqref{chunga.bis2}. 

Consider $\Phi_\Lambda(\alpha)=\Lambda\,e^{\alpha\|\textnormal{dist}(\cdot,\partial\Omega)\|_\infty}-\alpha$. It is elementary to show that $\Phi_\Lambda$ is convex and has a global minimum at
\[
\alpha_{\text{min}}=-\|\textnormal{dist}(\cdot,\partial\Omega)\|_\infty^{-1}\,\log\left(\Lambda\|\textnormal{dist}(\cdot,\partial\Omega)\|_\infty\right).
\]
This minimum value is
\[
\min_{\alpha\in\R}\Phi_\Lambda(\alpha)=\Phi_\Lambda(\alpha_\text{min})=\|\textnormal{dist}(\cdot,\partial\Omega)\|_\infty^{-1}
\big(
1+\log\left(\Lambda\|\textnormal{dist}(\cdot,\partial\Omega)\|_\infty\right)
\big).
\]
Whenever this minimum is strictly positive, equation \eqref{chunga.bis2} has no solution. This happens when $\Lambda>\Lambda_{\max}$ (in fact, Proposition \ref{prop:NO.exist.beyond.Lambda.hat} gives a stronger result in this case). Furthermore, notice that if $\Lambda=0$, then necessarily $\alpha=0$. These facts amount to ~$(iii)$. When the minimum equals  0, that is, when $\Lambda=\Lambda_{\max}$, then there exists a unique solution with $\alpha=\|\textnormal{dist}(\cdot,\partial\Omega)\|_\infty^{-1}$. This is part $(ii)$. And finally, for part $(i)$, notice that when the minimum is strictly negative ($0<\Lambda<\Lambda_{\max}$), equation \eqref{chunga.bis2} has two roots.
\end{proof}


\begin{remark}
Theorem \ref{conecone} yields the following implicit curve of cone solutions
\[
\Lambda_1(\Omega)\,\|u_\Lambda\|_\infty-\Lambda\,e^{\|u_\Lambda\|_\infty}=0,
\]
where $\Lambda_1(\Omega)=\|\textnormal{dist}(\cdot,\partial\Omega)\|_\infty^{-1}$ is the first $\infty$-eigenvalue, see \cite{Juutinen-Lindqvist-Manfredi}.
The same curve was deduced heuristically by Lions in the context of the Gelfand problem for the Laplacian in \cite[p. 465, item (h) and Remark 2.4]{Lions}. Unfortunately, Lions uses this example to caution against the heuristic reasoning since the bifurcation diagram is  of corkscrew-type for dimensions  $3\leq n \leq 9$.  One could wonder why we do not see a similar situation in Theorem \ref{conecone}.
However, according to \cite[Lemma 2.3]{DelPino}, the corresponding corkscrew-type diagram for the $p$-Laplacian in the radial case occurs in the range
\[
p<n<\frac{p\,(p+3)}{p-1},
\]
which cannot happen as $p\to\infty$.

\end{remark}


\subsection{Further non-existence results}
The following result shows that we can enlarge the region of nonexistence of solutions for certain domains by taking advantage of the curve of explicit solutions.

\begin{theorem}\label{asereje} 
Let  $\Omega$ be a bounded domain such that $\mathcal{M}\equiv\mathcal{R}$, and  assume $\mathcal{M}$ is \textit{Lipschitz connected}. Then, for every $\Lambda>0$, the only solutions of the problem
\begin{equation}\label{ecu.teo.uniqueness}
\left\{
\begin{aligned}
&\min\big\{|\nabla u_\Lambda(x)|-\Lambda\,e^{u_\Lambda(x)},
-\Delta_{\infty}u_\Lambda(x)\big\}=0&&\textrm{in}\ \Om,\\
&u_{\Lambda}=0&& \text{on}\ \partial\Omega
\end{aligned}
\right.
\end{equation}
 satisfying
\begin{equation}\label{curve.of.sols}
\Lambda_1(\Omega)\,\|u_\Lambda\|_\infty-\Lambda\,e^{\|u_\Lambda\|_\infty}\geq0,
\end{equation}
are the explicit solutions found in Theorem \ref{conecone}, which satisfy \eqref{curve.of.sols} with an equality.
\end{theorem}

The idea of  the proof of Theorem \ref{asereje}   is to show that any solution $u_\Lambda$ satisfying \eqref{curve.of.sols} must necessarily be a cone and therefore belong to the curve of solutions given by Theorem \ref{conecone}.
First, we  show that solutions of \eqref{ecu.teo.uniqueness} that satisfy \eqref{curve.of.sols} must lie below a cone with their same height.

\begin{lemma}\label{trururu}
Let $\Omega$ be a bounded domain and $u_{\Lambda}$ be a viscosity solution of \eqref{ecu.teo.uniqueness} satisfying \eqref{curve.of.sols}. Then, 
\[
u_{\Lambda}
\leq
\frac{\|u_\Lambda\|_\infty}{\|\textnormal{dist}(\cdot,\partial\Omega)\|_\infty}\,\textnormal{dist}(\cdot,\partial\Omega)\quad\textrm{in}\ \Omega.
\]

\end{lemma}

\begin{proof}
It is enough to prove that  
\begin{equation}\label{83746}
\min\left\{|\nabla{}u_\Lambda(x)|-\Lambda_1(\Omega)\,\|u_\Lambda\|_\infty,
-\Delta_{\infty}u_\Lambda(x)\right\}\leq0\quad\text{in}\ \Om
\end{equation}
in the viscosity sense. Then one gets $u_\Lambda(x)\leq
\|u_\Lambda\|_\infty\,\|\textnormal{dist}(\cdot,\partial\Omega)\|_\infty^{-1}\,\textnormal{dist}(x,\partial\Omega)$ in $\Omega$ by comparison (Lemma \ref{comparison.f(x)}), and the result follows.

 To prove \eqref{83746}, let  $\phi\in{C}^2$ such that $u_{\Lambda}-\phi$ has a maximum
at $x_0\in\Omega$. As $u_{\Lambda}$ is a viscosity solution of \eqref{ecu.teo.uniqueness}, it satisfies
\[
\min\left\{|\nabla{}\phi(x_0)|-\Lambda\,e^{u_{\Lambda}(x_0)},
-\Delta_{\infty}\phi(x_0)\right\}\leq0\quad\textrm{in}\ \Om.
\]
If $-\Delta_\infty \phi(x_0) \leq 0$ we are done, so assume $-\Delta_{\infty}\phi(x_0)>0$ and $|\nabla{}\phi(x_0)|-\Lambda\,e^{u_{\Lambda}(x_0)}\leq0$. Using \eqref{curve.of.sols}, we have
\[
|\nabla{}\phi(x_0)|-\Lambda_1(\Omega)\,\|u_\Lambda\|_\infty\leq{}
\Lambda\,e^{u_{\Lambda}(x_0)}-\Lambda_1(\Omega)\,\|u_\Lambda\|_\infty\leq0,
\]
 and then
\[
\min\left\{|\nabla{}\phi(x_0)|-\Lambda_1(\Omega)\,\|u_\Lambda\|_\infty,
-\Delta_{\infty}\phi(x_0)\right\}\leq0\quad\textrm{in}\ \Om
\]
as desired.
\end{proof}

\begin{remark}
Lemma \ref{trururu} holds for any bounded domain
$\Omega$ without the assumption $\mathcal{M}\equiv\mathcal{R}$.
\end{remark}

Next, we recall the following result from \cite[Theorem 2.4, (i)]{Yu}, which is a crucial point in the proof of Theorem \ref{asereje}.

\begin{lemma}\label{lemma.yu}
Let $\Omega$ be a bounded domain such that $\mathcal{M}$ is Lipschitz connected. If $u$ is $\infty$-superharmonic (see \cite{lindqvist-manfredi1,lindqvist-manfredi2}) then,
\[
\big\{x\in\Omega:\
u(x)=\|u\|_{L^\infty(\Omega)}\big\}\equiv\mathcal{M}.
\]
\end{lemma}

Now, we can complete the proof of Theorem \ref{asereje}.

\begin{proof}[Proof of Theorem \ref{asereje}]
Consider $u_\Lambda$  solution of \eqref{ecu.teo.uniqueness}  satisfying \eqref{curve.of.sols}. 
Notice that 
\[
v(x)=\frac{\|u_\Lambda\|_\infty}{\|\textnormal{dist}(\cdot,\partial\Omega)\|_\infty}\,\textnormal{dist}(\cdot,\partial\Omega)
\]
 is the unique (see \cite{Jensen93})  viscosity solution of the problem
\begin{equation}\label{homersaysdoh2}
\left\{
\begin{split}
-&\Delta_{\infty}v(x)=0&&\text{in}\ \Omega\setminus\mathcal{M}\\
& v(x)=\|u_\Lambda\|_\infty&&\text{on}\ \mathcal{M}\\
& v(x)=0&&\text{on}\ \partial\Omega.
\end{split}
\right.
\end{equation}
Since $u_{\Lambda}$ is $\infty$-superharmonic, it is also a viscosity supersolution of \eqref{homersaysdoh2} by  Lemma \ref{lemma.yu}. Then, we get  $v\leq u_{\Lambda}$ by comparison (see \cite{Jensen93}), and Lemma  \ref {trururu} yields $u_{\Lambda} \equiv v$. That is, $u_\Lambda$ is of the form \eqref{cono.bis2}.
Since all the solutions of \eqref{ecu.teo.uniqueness} of the form \eqref{cono.bis2} are given by Theorem \ref{conecone}, we find that there are no solutions with
\[
\Lambda_1(\Omega)\,\|u_\Lambda\|_\infty-\Lambda\,e^{\|u_\Lambda\|_\infty}>0.
\]
Furthermore, if $\Lambda_1(\Omega)\,\|u_\Lambda\|_\infty-\Lambda\,e^{\|u_\Lambda\|_\infty}=0$, then $u_\Lambda$ must be one of the explicit solutions  in Theorem \ref{conecone}.
\end{proof}

\


\begin{thebibliography}{99}


\bibitem{Abdellaoui.Peral.2003}
B. Abdellaoui, I. Peral; \emph{Existence and nonexistence results for quasilinear elliptic equations involving the $p$-Laplacian with a critical potential}, Ann. Mat. Pura Appl. 182 (2003), pp. 247--270.



\bibitem{Anane}A. Anane; \emph{Simplicit\`e et isolation de la premiere valeur propre du $p$-laplacien avec poids}, C. R. Acad. Sci. Paris S\'{e}r. I Math. 305 (1987), no. 16,   725--728.

\bibitem{Bha-DiBe-Man} T. Bhattacharya, E. DiBenedetto, J. Manfredi; \emph{Limit as $p\rightarrow\infty$ of $\Delta_{p}u_p=f$ and related extremal problems}, Rend. Sem. Mat. Univ. Politec. Torino (1989), pp. 15-68.

\bibitem{Bratu}G. Bratu; \emph{Sur les \'equations int\'egrales non lin\'eaires}, Bulletin de la Soci\'et\'e Math\'ematique de France 42 (1914): 113--142.

\bibitem{Brezis-Oswald}H. Brezis, L. Oswald; \emph{Remarks on sublinear elliptic equations}, Nonlinear Analysis, Theory, Methods \& Applications, Vol. 10 (1986), no. 1, pp. 55--64.


\bibitem{Cabre.Sanchon.2007} X. Cabr\'e, M. Sanch\'on; \emph{Semi-stable and extremal solutions of reaction equations involving the $p$-Laplacian}, Communications on Pure \& Applied Analysis, 2007, 6 (1), pp. 43--67.

\bibitem{Chandrasekhar}  S. Chandrasekhar; \emph{An introduction to the study of stellar structures} Dover, New York (1957).

\bibitem{Chanillo-Kiessling} S. Chanillo, M. Kiessling; \emph{Surfaces with prescribed Gauss curvature}, Duke Math. J. 105 (2000), no. 2, 309--353.

\bibitem{Charro.Parini}F. Charro, E. Parini; \emph{Limits as $p\to\infty$ of $p$-laplacian  problems with a superdiffusive power-type nonlinearity: positive and sign-changing solutions}, J. Math. Anal. Appl. 372 (2010), no. 2, 629--644.

\bibitem{Charro.Parini2}F. Charro, E. Parini; \emph{Limits as $p\to\infty$ of $p$-laplacian eigenvalue problems perturbed with a concave or convex term}, Calc. Var. and PDE 46 (2013), no. 1-2,  403--425.

\bibitem{Charro.Peral}F. Charro, I. Peral; \emph{Limit branch of solutions as $p\rightarrow\infty$ for a family of sub-diffusive  problems related to the p-laplacian} Comm. Partial Differential Equations, vol. 32 (2007), no. 12, pp. 1965 - 1981.

\bibitem{Charro.Peral2}F. Charro, I. Peral; \emph{Limits as $p\to\infty$ of $p$-Laplacian concave-convex problems}, Nonlinear Analysis: Theory, Methods \& Applications 75, no. 4 (2012): 2637-2659.



\bibitem{CIL} M. G. Crandall, H. Ishii, P. L. Lions; \emph{User's Guide to Viscosity Solutions of Second Order Partial Differential Equations}, Bull.  Amer. Math. Soc. 27 (1992), no. 1,  pp. 1-67.


\bibitem{Davila}J. D\'avila; \emph{Singular solutions of semi-linear elliptic problems}, Handbook of differential equations: stationary partial differential equations. Vol. VI, Handb. Differ. Equ., Elsevier/North- Holland, Amsterdam, 2008, pp. 83--176.

\bibitem{DelPino} M. Del Pino, J. Dolbeault, M. Musso; \emph{Multiple Bubbling for the exponential nonlinearity in the slightly supercritical case}, Comm. on Pure and Applied Analysis 5, no. 3 (2006) pp. 463--482. 


\bibitem{Endem.1907} R. Emden; \emph{Gaskugeln: Anwendungen der mechanischen W\"armetheorie auf kosmologische und meteorologische Probleme (Gas balls: applications of mechanical heat theory to cosmological and meteorological problems)}. Germany: B.G. Teubner, 1907.



\bibitem{Fukagai} N. Fukagai, M. Ito, K. Narukawa; \emph{Limit as $p\rightarrow\infty$ of $p$-Laplace eigenvalue problems and $L^\infty$-inequality of the Poincare type}, Differential Integral Equations 12 (1999), no. 2, pp. 183--206.

\bibitem{GP} J. Garc\'{\i}a Azorero, I. Peral; \emph{Existence and nonuniqueness for the p-laplacian: Nonlinear eigenvalues}, Comm. Partial Differential Equations, vol 12, no. 12 (1987)  1389-1430.

\bibitem{GP2} J. Garc\'{\i}a Azorero, I. Peral Alonso; \emph{On a Emden-Fowler type equation}, Nonlinear Analysis T.M.A., 18, no. 11 (1992), pp. 1085--1097.

\bibitem{GPP} J. Garc\'{\i}a Azorero, I. Peral Alonso, J.P. Puel; \emph{Quasilinear problems with exponential growth in the reaction term}, Nonlinear Analysis T.M.A., 22, no. 4 (1994), pp. 481--498. 


\bibitem{Gelfand}I. M. Gel'fand; \emph{Some problems in the theory of quasilinear equations}, Amer. Math. Soc. Transl. (2) 29 (1963), 295--381.

\bibitem{Grosjean}
J.-F. Grosjean; \emph{$p$-Laplace operator and diameter of manifolds},  Ann. Global Anal. Geom., 2005, 28, pp.257-270. 

\bibitem{Jacobsen.Schmitt.2002} J. Jacobsen, K. Schmitt; \emph{The Liouville-Bratu-Gelfand problem for radial operators}, Journal of Differential Equations 184, no. 1 (2002), pp. 283--298.

\bibitem{Jensen93} R. Jensen; \emph{Uniqueness of Lipschitz extensions: Minimizing the sup norm of the gradient}, Arch. Rational Mech. Anal. 123 (1993),   51--74.

\bibitem{Joseph.Lundgren.1973} D. D. Joseph, T. S. Lundgren; \emph{Quasilinear Dirichlet problems driven by positive sources}, Archive for Rational Mechanics and Analysis 49.4 (1973), pp. 241--269.

\bibitem{JuutinenTESIS} P. Juutinen; \emph{Minimization problems for Lipschitz functions via viscosity solutions}. Dissertation, University of Jyv\"{a}skul\"{a}, Jyv\"{a}skul\"{a}, 1998. Ann. Acad. Sci. Fenn. Math. Diss. No. 115 (1998), 53 pp.

\bibitem{Juutinen} P. Juutinen; \emph{Principal eigenvalue of a badly degenerate operator}, J. Differential Equations 236 (2007), no. 2, 532--550.

\bibitem{Juutinen-Lindqvist} P. Juutinen, P. Lindqvist;  \emph{On the higher eigenvalues for the $\infty$-eigenvalue problem}, Calc. Var. and Partial Differential Equations, 23:169--192, 2005.

\bibitem{Juutinen-Lindqvist-Manfredi} P. Juutinen, P. Lindqvist, J. Manfredi, \emph{The $\infty$-eigenvalue problem}, Arch. Ration. Mech. Anal. 148 (1999), no. 2, pp. 89-105.


\bibitem{Kawohl90} B. Kawohl; \emph{On a family of torsional creep problems}, J. Reine Angew. Math. 410 (1990),   1--22.



\bibitem{lindqvist90} P. Lindqvist; \emph{On the equation ${\rm div}\,(\vert \nabla u\vert \sp {p-2}\nabla u)+\lambda\vert u\vert \sp {p-2}u=0$}, Proc. Amer. Math. Soc. 109 (1990), no. 1, 157--164.

\bibitem{lindqvist92} P. Lindqvist; \emph{Addendum: ``On the equation ${\rm div}(\vert \nabla u\vert \sp {p-2}\nabla u)+\lambda\vert u\vert \sp {p-2}u=0$'' [Proc. Amer. Math. Soc. 109 (1990), no. 1, 157--164; MR1007505 (90h:35088)]}, Proc. Amer. Math. Soc. 116 (1992), no. 2,   583--584.

\bibitem{lindqvist-manfredi1} P. Lindqvist, J. Manfredi; \emph{The Harnack inequality for $\infty$-harmonic functions}, Elec. J. Diff. Eqs. 5 (1995),   1-5.

\bibitem{lindqvist-manfredi2} P. Lindqvist, J. Manfredi; \emph{Note on $\infty$-superharmonic functions}, Revista Matem\'{a}tica de la Universidad Complutense de Madrid 10 (1997),   1-9.


\bibitem{Lions} P. L. Lions; \emph{On the Existence of Positive Solutions of Semilinear Elliptic Equations}, SIAM Review 24, No. 4 (1982), pp. 441-467.



\bibitem{Liouville}J. Liouville; \emph{Sur l'\'equation aux diff\'erences partielles $\frac{d^2\log\lambda}{dudv}\pm \frac{\lambda}{2a^2} =0$}, Journal de math\'ematiques pures et appliqu\'ees 1re s\'erie, tome 18 (1853), p. 71--72.




\bibitem{Mihailescu-StancuDumitru-Varga} M. Mih\u{a}ilescu, D. Stancu-Dumitru, C. Varga; \emph{The convergence of nonnegative solutions for the family of problems $-\Delta_p u=\lambda e^u$ as $p\to\infty$}, ESAIM: Control, Optimisation and Calculus of Variations 24 (2) (2018), 569--578.

\bibitem{IP} I. Peral; \emph{ Some results on Quasilinear Elliptic Equations: Growth versus Shape},   153-202, {\it Proceedings of the Second School of Nonlinear Functional Analysis and Applications to Differential Equations} I.C.T.P. Trieste, Italy, A. Ambrosetti and {it alter} editors. World Scientific, 1998.

\bibitem{Richardson.1921}O. W. Richardson; \emph{The Emission of Electricity from Hot Bodies}, India:Longmans, Green and Company, 1921.

\bibitem{Sanchon.2007}M. Sanch\'on; \emph{Regularity of the extremal solution of some nonlinear elliptic problems involving the $p$-Laplacian}, Potential Analysis 27, no. 3 (2007), pp. 217-224.



\bibitem{Walker} G. W. Walker; \emph{Some problems illustrating the forms of nebulae}, Proceedings of the Royal Society of London. Series A, Containing Papers of a Mathematical and Physical Character 91.631 (1915): 410--420.

\bibitem{Yu} Y. Yu; \emph{Some properties of the ground states of the infinity Laplacian}, Indiana Univ. Math. J. 56 (2007), no. 2,   947--964.


\end{thebibliography}
\end{document}